\documentclass[11pt, letterpaper]{amsart} 

\usepackage{amscd,amssymb,amsopn,amsmath,amsthm, graphics,amsfonts, enumerate,verbatim,calc,
stmaryrd}
\usepackage[dvips]{graphicx}
\usepackage{url}
\usepackage{mathrsfs} 
\usepackage[all]{xy}
\usepackage{mathdots}

\usepackage{color}

\usepackage{mathtools} 
\usepackage{colonequals} 

\usepackage{tikz-cd}

\usepackage[utf8]{inputenc}
\usepackage[T1]{fontenc}
\usepackage{lmodern}
\usepackage[english]{babel}
\usepackage[autostyle]{csquotes}

\usepackage{xcolor}
\usepackage{hyperref}
\hypersetup{
  bookmarksnumbered=true, %
  pdftitle={}, 
  pdfsubject={}, 
  pdfauthor={}, 
  pdfkeywords={}, 
  colorlinks=true, 
  linkcolor=black,
  citecolor=black,
  urlcolor=black,
  bookmarksnumbered 
}
\RequirePackage{cleveref} 

\newcommand{\citeSta}[1]{\cite[Tag \href{https://stacks.math.columbia.edu/tag/#1}{#1}]{Stacks}}

\newtheorem{theorem}{Theorem}[section]
\newtheorem{lemma}[theorem]{Lemma}
\newtheorem{proposition}[theorem]{Proposition}
\newtheorem{corollary}[theorem]{Corollary}
\newtheorem{maintheorem}{Main Theorem}

\newtheorem{conjecture}[theorem]{Conjecture}

\theoremstyle{definition}
\newtheorem{definition}[theorem]{Definition}
\newtheorem{remark}[theorem]{Remark}
\theoremstyle{remark}

\newtheorem{discussion}[theorem]{Discussion}
\newtheorem{example}[theorem]{Example}
\newtheorem{acknowledgement}{Acknowledgement}

\newcommand{\Ker}{\operatorname{Ker}}
\newcommand{\Spec}{\operatorname{Spec}}

\newcommand{\V}{\operatorname{V}}

\newcommand{\id}{\operatorname{id}}
\newcommand{\Ext}{\operatorname{Ext}}

\newcommand{\Frac}{\operatorname{Frac}}

\newcommand{\tor}{\operatorname{tor}}

\newcommand{\Proj}{\operatorname{Proj}}

\newcommand{\Frob}{\operatorname{Frob}}

\newcommand{\colim}{\operatorname{colim}}

\newcommand\omicron{o}
\newcommand{\prarrow}[2]{\ar@<0.5ex>[r]^-{#1}\ar@<-0.5ex>[r]_-{#2}} \newcommand{\plarrow}[2]{\ar@<0.5ex>[l]^-{#1}\ar@<-0.5ex>[l]_-{#2}} \newcommand{\pdarrow}[2]{\ar@<0.5ex>[d]^-{#1}\ar@<-0.5ex>[d]_-{#2}} \newcommand{\puarrow}[2]{\ar@<0.5ex>[u]^-{#1}\ar@<-0.5ex>[u]_-{#2}}

\newcommand{\vpl}{\operatornamewithlimits{\varprojlim}}

\newcommand{\fm}{\mathfrak{m}}

\newcommand{\perfd}{\mathrm{perfd}}

\begin{document}
\title[$\delta$-rings, perfectoid towers, and lim Cohen--Macaulay sequences]
{$\delta$-rings, perfectoid towers, and lim Cohen--Macaulay sequences}

\author[S. Ishiro]{Shinnosuke Ishiro}
\address{National Institute of Technology, Gunma College, 580 Toriba-machi, Maebashi-shi, Gunma 371-8530, Japan}
\email{shinnosukeishiro@gmail.com}

\author[K. Shimomoto]{Kazuma Shimomoto}
\address{Department of Mathematics, Institute of Science Tokyo, 2-12-1 Ookayama, Meguro, Tokyo 152-8551, Japan}
\email{shimomotokazuma@gmail.com}




\begin{abstract}
The aim of this article is to study basic structures and interrelations of $\delta$-rings, perfectoid towers, and lim Cohen--Macaulay sequences over Noetherian rings in positive or mixed characteristic. We also discuss the deformation of perfectoid purity via perfectoid towers. In the latter part of this paper, we discuss some methods for constructing perfectoid towers, dealing with $p$-torsion-free and $p$-torsion cases, respectively. Some interesting examples arise as quotients by monomial or binomial ideals or determinantal rings. We also explain a geometric method with a view toward constructing rings with certain singularities.
\end{abstract}

\maketitle

\setcounter{tocdepth}{3}
\tableofcontents

\section{Introduction}

The aim of this paper is to provide some methods for the construction of perfectoid towers and study their relations with $\delta$-rings and lim Cohen--Macaulay sequences. 
The notion of perfectoid towers was introduced in \cite{INS23} as a Noetherian approximation of perfectoid rings with applications to singularities of algebraic varieties. 
The existence of such a tower over a Noetherian ring $R$ enables us to deduce some interesting arithmetic information via tilting correspondence. The most striking application is the study of purity for the Brauer group and flat cohomology as done in \cite{CK19} and \cite{KS20}. The notion of lim Cohen--Macaulay sequences was introduced and studied by Bhatt--Hochster--Ma in \cite{BHM24} with potential applications to the positivity conjecture of the intersection multiplicities of Serre.
The notion of lim Cohen--Macaulay sequences (see \Cref{LimCMDef}) is a substitute for small (or maximal) Cohen--Macaulay modules.
In fact, we are motivated by the following conjecture (see \cite[Theorem 6.1]{Ho17} for a precise statement):

\begin{conjecture}[Bhatt--Hochster--Ma]
Every Noetherian complete local domain whose residue field is algebraically closed has a lim Cohen--Macaulay sequence. If this conjecture is true, then Serre's positivity conjecture on multiplicities holds in general.
\end{conjecture}

Motivated by this conjecture, we prove a fundamental result in connection with lim Cohen--Macaulay sequences asserting that perfectoid towers are a source of producing non-trivial examples of lim Cohen--Macaulay sequences. See Definition \ref{invqperf} for notation.

\begin{maintheorem}[\Cref{weakperfectlimCM}]
\label{Mtheorem2}
Let $(R,\fm,k)$ be a complete Noetherian local domain of mixed characteristic with perfect residue field of characteristic $p>0$. 
Suppose that there exists a perfectoid tower $( \{ R_i \}_{i \geq 0}, \{t_i\}_{i \geq 0})$ arising from $(R, I_0)$ for some ideal $I_0 \subseteq R$. Then $\{ R_i \}_{i \geq 0} $ is a lim Cohen--Macaulay sequence of algebras.
\end{maintheorem}

The key to the proof of Main Theorem \ref{Mtheorem2} is that the tilt of a perfectoid tower is a lim Cohen--Macaulay sequence.
By virtue of \Cref{Mtheorem2}, the existence of perfectoid towers over any complete Noetherian local domain becomes a significant question for solving the positivity conjecture of Serre's intersection multiplicity. Moreover, it seems that there is no systematic way in producing perfectoid towers over Noetherian local rings. In other words, their existence is present to us in different guises, making this quest worthwhile. Let us explain methodology in what follows.

In order to give another application of the existence of perfectoid towers, we explore the deformation problem of perfectoid purity and provide many examples which are Gorenstein but not complete intersections.
In \cite{BMPSTWW24}, \textit{Perfectoid pure singularities} are introduced as a mixed characteristic analogue of $F$-pure singularities and their fundamental properties are established.
Examples of perfectoid pure singularities are also provided under the assumption that a given ring is locally complete intersection.
The deformation problem of perfectoid purity first appeared in \cite[Theorem 6.6]{BMPSTWW24} as an inversion of adjunction-type theorem.
As a special case of a deformation of perfectoid purity, the following conjecture is proposed in \cite{BDSW25}.

\begin{conjecture}[{\cite[Conjecture 1.1]{BDSW25}}]\label{ConjectureDeformPerfdpurity}
    Let $(R,\fm,k)$ be a $p$-torsion-free complete local Gorenstein ring in mixed characteristic $(0,p)$.
    If $R/pR$ is an $F$-pure domain, then $R$ is perfectoid pure.
\end{conjecture}

They also proposed other conjectures related to Conjecture \ref{ConjectureDeformPerfdpurity}.
If readers are interested in the related conjectures, see \cite{BDSW25}.
In this paper, we prove that Conjecture \ref{ConjectureDeformPerfdpurity} holds if a given Gorenstein local ring admits a perfectoid tower.

\begin{maintheorem}[\Cref{DeformPerfdPurity}]\label{MainTheorem3}
    Let $(R,\fm,k)$ be a Gorenstein local ring whose residue field is $F$-finite.
    Assume that there is a perfectoid tower $(\{ R_i \}_{i \geq 0}, \{ t_i \}_{i \geq 0})$ arising from some pair $(R, I_0)$ such that $R$ is $I_0$-torsion-free.\footnote{By definition, the ideal $I_0$ is principal and contains $p$. Hence this theorem covers the setting of \Cref{ConjectureDeformPerfdpurity}.}
    Suppose that $\overline{R} =R/I_0$ is $F$-split. 
    Then $R$ is perfectoid pure.
\end{maintheorem}

In the final part of this paper, we discuss the construction of perfectoid towers by using Frobenius lifts.
The key to this construction is a $\phi$-stable ideal, which is an ideal $I$ of a $\delta$-ring $(A,\delta)$ with the associated Frobenius lift $\phi$ satisfying the inclusion $I^{[\phi]} = \langle\phi(a) ~|~a \in I\rangle \subseteq I$.
In general, for a $\delta$-ring $A$ with the associated Frobenius lift $\phi$ and an ideal $I$, the quotient ring $R:=A/I$ does not necessarily admit the Frobenius lift induced by $\phi$.
However, if $I$ is $\phi$-stable, the Frobenius lift $\phi$ induces the ring map $\phi_{R}$ on $R$.
We define an analogue of the finite perfection $R_i$ as a finite colimit $\colim\{R \xrightarrow{\phi_{R}} R \xrightarrow{\phi_{R}}\cdots\xrightarrow{\phi_{R}}R \}$.

The third main result of this paper is to establish a method of constructing perfectoid towers arising from some $p$-torsion-free Noetherian local rings.

\begin{maintheorem}[{\Cref{pTorFreePerfdTowers}}]\label{MainTheorem2}
    Let $(V, (p), k)$ be a complete unramified discrete valuation ring in mixed characteristic with perfect residue field $k$. 
    Let $(A,\delta)$ be a $\delta$-ring with the associated Frobenius lift $\phi$ and let $I$ be a $\phi$-stable ideal.
    We set $R := A/I$.
    Suppose that the following conditions hold:
    \begin{itemize}
        \item $(R,\fm,k)$ is a Noetherian local $V$-algebra whose residue field is $k$.
        \item $R$ is $p$-torsion-free.
        \item $\overline{R} :=R/pR$ is reduced.
    \end{itemize}
    Also, we set $V_i := V[p^{1/p^i}]$ and the inclusion map $t'_i : V_i \to V_{i+1}$ for any $i \geq 1$.
    Then the following assertions hold.
    \begin{enumerate} 
\item
The tower $(\{S_i\}_{i \geq 0}, \{ s_i \}_{i \geq 0}) := (\{ R_i \otimes_V V_i\}_{i \geq 0}, \{(\tau_R)_i \otimes_V t_i'\}_{i \geq 0})$ is a perfectoid tower arising from $(R,(p))$, where $(\tau_R)_i : R_i \to R_{i+1}$ is the canonical ring map.

\item
The tilt $(\{S_i\}_{i \geq 0}, \{ s_i \}_{i \geq 0})$ is isomorphic to the tower
\[
\xymatrix{
\overline{R_0}\otimes_{\overline{V}} V^{s.\flat} \ar[r] & \overline{R_1} \otimes_{\overline{V}} V_1^{s.\flat} \ar[r] & \overline{R_2} \otimes_{\overline{V}} V_2^{s.\flat} \ar[r] & \cdots.
}
\]
where $\overline{V}$ is the residue field $k$ of $V$.
\end{enumerate}
\end{maintheorem}

Incidentally, a similar construction given in \Cref{MainTheorem2} is provided in \cite{I23}.
While the construction given in \cite{I23} relies on the method of prisms established by Bhatt and Scholze together with $\delta$-structures of rings, we focus on the stability of Frobenius lifts.

\subsection{Computational results}
By applying \Cref{MainTheorem2}, we construct perfectoid towers for several classes of $p$-torsion-free Noetherian local rings. 
For example, a quotient ring $W(k)\llbracket x_2,\ldots, x_d \rrbracket /I$, where $I$ is an ideal generated by squarefree monomials or binomials, admits a perfectoid tower.
For details, see \Cref{ExamplePerfdTowerp-torfree}.
In particular, we construct perfectoid towers arising from several determinantal rings defined by $2$-minors.
Some ladder determinantal rings have a perfectoid tower in the same manner as above.
Singularities of ladder determinantal rings in positive characteristic are already well-studied.
Indeed, a sufficient condition for a ladder determinantal ring to be $F$-pure was provided by Conca and Herzog \cite{CH97}, and De Stefani--Monta{\~n}o--Betancourt--Seccia--Varbaro recently proved that any ladder determinantal ring in positive characteristic is $F$-pure \cite{DSMNBSV26}.
A criterion for a ladder determinantal ring to be Gorenstein (resp. a complete intersection) was given by Conca \cite{Con95} (resp. Glassbrenner and Smith \cite{GS95}).
By combining their results with our main theorems, we obtain examples of perfectoid pure singularities, which are Gorenstein but not a complete intersection.
We also mention a geometric method. 
In \cite{IS24}, a deformation theoretic technique is used to find some new classes of projective varieties that are Frobenius-liftable over the ring of Witt vectors. Using these varieties, we show that the associated ring of sections provides a normal and non-Cohen-Macaulay ring with a Frobenius lift (see \Cref{delta-structure}).

At the end of this paper, we will also provide some example of a perfectoid tower with $p$-torsion elements. This example is constructed from the quotient ring of monomial ideals generated by squarefree monomials such as $p^{\epsilon_1}x_2^{\epsilon_2} \cdots x_d^{\epsilon_d}$ in $\mathbb{Z}_p\llbracket x_2,\ldots, x_d \rrbracket$ for $\epsilon_i \in \{0, 1\}$.
This example is also constructed by using a Frobenius lift on $\mathbb{Z}_p\llbracket x_2,\ldots, x_d \rrbracket$, even though the resulting quotient ring is not a $\delta$-ring (see \Cref{egnon-deltaring}).

\subsection{Convention}
Throughout this paper, we follow the convention stated below. 
We fix a prime $p>0$. 
All rings are assumed to be commutative.
We say that a \textit{tower of rings} is a direct system of rings indexed by non-negative integers. 
We denote a tower of rings by $(\{R_i\}_{i \geq 0},\{t_i\}_{i \geq 0})$ or $R_0 \xrightarrow{t_0} R_1 \xrightarrow{t_1} R_2 \xrightarrow{t_2} \cdots$, where $t_i : R_i \rightarrow R_{i+1}$ is a ring map.

\begin{acknowledgement}
The authors are sincerely grateful to Yves Andr\'e for pointing out the relation between perfectoid towers and lim Cohen--Macaulay sequences.
His comment served as the starting point of this research.
The authors are also grateful to Ryo Ishizuka for many fruitful discussions, particularly in the establishment of \Cref{MainTheorem3}.
The authors would like to thank Kazuki Hayashi for his careful reading and many helpful comments.
The authors would also like to thank Kei Nakazato for his advice.
The first author is deeply grateful to Ken-ichi Yoshida for instructive discussions on monomial ideals in regular local rings.
\end{acknowledgement}

\section{Basics on perfectoid towers}

In this section, we review the definition of perfectoid towers and their tilts, and collect properties that will be used in the subsequent sections. 
First, let us recall perfect towers and ($p$-)purely inseparable towers from \cite{INS23}.
\begin{definition}[Perfect towers]
A \emph{perfect tower} is a tower of rings that is isomorphic to a tower:  
\begin{equation*}
\xymatrix{
R \ar[r]^{F_{R}} & R \ar[r]^{F_{R}} & R \ar[r]^{F_R} & \cdots, 
}
\end{equation*}
where $R$ is a reduced $\mathbb{F}_p$-algebra and $F_R$ is the Frobenius endomorphism on $R$.
\end{definition}

\begin{definition}
\label{invqperf}
Let $R$ be a ring, let $I_0\subseteq R$ be an ideal, and let $p>0$ be a prime number. A tower $(\{R_i\},\{t_i\})_{i \geq 0}$ is called a \textit{$p$-purely inseparable tower (or purely inseparable tower, for short) arising from $(R, I_0)$} if it satisfies the following axioms.
\begin{itemize}
\item[(a)]
$R_{0}=R$ and $p\in I_0$. 
\item[(b)]
For any $i \geq 0$, the ring map $\overline{t_i}: R_i/I_0R_i\to R_{i+1}/I_0R_{i+1}$ induced by $t_i$ is injective. 
\item[(c)]
For any $i\geq 0$, the image of the Frobenius endomorphism on $R_{i+1}/I_0R_{i+1}$ is contained in the image of $\overline{t_{i}}: R_{i}/I_0R_{i}\to R_{i+1}/IR_{i+1}$.  
\end{itemize}
\end{definition}

In the following, we denote $R_i/I_0{R_i}$ by $\overline{R_i}$ for a purely inseparable tower $(\{R_i\},\{t_i\})_{i \geq 0}$. 
For a purely inseparable tower $(\{R_i\}, \{t_i\})_{i \geq 0}$ arising from $(R,I_0)$, by the axioms (b) and (c), there exists a ring map $F_i : \overline{R_{i+1}} \to \overline{R_i}$ uniquely such that the following diagram commutes:
\[
\xymatrix{
R_{i+1}/I_0 R_{i+1}\ar@{->}[rrd]_{F_i}\ar@{->}[rr]^{F_{\overline{R_{i+1}}}}&&R_{i+1}/I_0 R_{i+1}\\
 &&R_{i}/I_0 R_{i}\ar[u]^{\overline{t_{i}}},
}
\]
where $F_{\overline{R_{i+1}}}$ is the Frobenius endomorphism on $\overline{R_{i+1}}$.
We call $F_i$ \textit{the $i$-th Frobenius projection} of $(\{R_i\}, \{t_i\})_{i \geq 0}$ associated to $(R,I_0)$.

\begin{definition}{(Perfectoid towers)}\label{DefPerfectoidTowers}
Let $R$ be a ring, and let $I_{0}\subseteq R$ be an ideal. 
\begin{enumerate}
\item
A tower $(\{R_{i}\}_{i\geq 0}, \{t_{i}\}_{i\geq 0})$ is called a \textit{(p-)perfectoid tower arising from} $(R, I_{0})$ if it is a $p$-purely inseparable tower arising from $(R,I_0)$ and satisfies the following additional axioms.
\begin{itemize}
\item[(d)]
For every $i\geq 0$, the $i$-th Frobenius projection $F_i : R_{i+1}/I_{0}R_{i+1} \to R_{i}/I_{0}R_{i}$ is surjective. 
\item[(e)]
For every $i \geq 0$, $R_i$ is an $I_0$-adically Zariskian ring (i.e.\ $I_0R_i$ is contained in the Jacobson radical of $R_i$).
\item[(f)]
$I_{0}$ is a principal ideal, and $R_{1}$ contains a principal ideal $I_{1}$ that satisfies the following axioms. 
\begin{itemize}
\item[(f-1)]
$I_1^p = I_0R_1$. 
\item[(f-2)] 
For every $i \geq 0$, $\ker (F_i)=I_1(R_{i+1}/I_{0}R_{i+1})$.
\end{itemize}
\item[(g)]
For every $i\geq 0$, $I_{0}(R_{i})_{I_{0}\textnormal{-tor}}=(0)$. Moreover, there exists a (unique) bijective map 
$(F_{i})_{\textnormal{tor}}: (R_{i+1})_{I_{0}\textnormal{-tor}}\to (R_{i})_{I_{0}\textnormal{-tor}}$ such that the diagram: 
\begin{equation}\label{24519N}
\vcenter{
\xymatrix{
(R_{i+1})_{I_0\textnormal{-tor}}\ar[d]_{(F_i)_{\textnormal{tor}}}\ar[r]^{\varphi_{I_{0}, R_{i+1}}}&R_{i+1}/I_{0}R_{i+1}\ar[d]^{F_{i}}\\
(R_{i})_{I_{0}\textnormal{-tor}}\ar[r]_{\varphi_{I_{0}, R_i}}&R_{i}/I_{0}R_{i}
}}
\end{equation}
commutes, where $\varphi_{I_0, R_i}: (R_i)_{I_0\textnormal{-tor}}\to R_i/I_0R_i$ is the composition of the inclusion map $(R_i)_{I_0\textnormal{-tor}} \hookrightarrow R$ and the projection $R_i \twoheadrightarrow R_i/I_0R_i$.
We note that this map is injective (see \cite[Corollary 3.19.]{INS23})
\end{itemize}
\item 
For a perfectoid tower $(\{ R_i \}_{i \geq 0}, \{ t_i\}_{i \geq 0})$, the \textit{tilt of the perfectoid tower} $(\{ (R_i)^{s.\flat}_{I_0} \}_{i \geq 0}, \{ (t_i)^{s.\flat}_{I_0} \}_{i \geq 0})$ is defined as the tower whose layers and transition maps are given as follows:
\begin{itemize}
    \item $(R_i)_{I_0}^{s.\flat} := \vpl \{ \cdots \xrightarrow{F_{i+1}} R_{i+1}/I_0R_{i+1} \xrightarrow{F_i} R_i/I_0R_i  \}$.
    \item $(t_i)_{I_0}^{s.\flat} : R_i^{s.\flat} \to R_{i+1}^{s.\flat}\ ; (\overline{a_j})_{j \geq i} \mapsto (\overline{t_{i+j}}(\overline{a_j}))_{j \geq i}$.
\end{itemize}
The ring $(R_i)^{s.\flat}_{I_0}$ is called the \textit{$i$-th small tilt associated to $(R_0, I_0)$}.
If no confusion occurs, we denote the tilt by $(\{R_i^{s.\flat} \}_{i \geq 0}, \{ t_i^{s.\flat} \}_{i \geq 0})$.
Also, the ideal $I_0^{s.\flat}$ is defined as the kernel of the $0$-th projection $R_0^{s.\flat} \to R_0/I_0$.
\end{enumerate}
\end{definition}

By the definition of $I_0^{s.\flat}$, we obtain the isomorphism
\begin{equation}\label{PerfdTiltIsom}
    R_i^{s.\flat}/I_0^{s.\flat}R_i^{s.\flat} \xrightarrow{\cong}R_i/I_0R_i.
\end{equation}
This isomorphism is one of the most effective tools in the theory of perfectoid towers.

\begin{remark}
    Due to recent work of \cite{BMS18}, \cite{GR23}, and \cite{INS23}, there are several characterizations of perfectoid rings.
    The axiom (g) in \Cref{DefPerfectoidTowers} is based on the characterization provided in \cite[Theorem 3.50]{INS23}.
\end{remark}

We summarize some important properties of perfectoid towers and their tilts, which are needed in this paper.
We refer the reader to \cite{INS23} for their proofs.

\begin{lemma}\label{PropPerfdTower}
Let $( \{ R_i \}_{ i \geq 0}, \{ t_i \}_{i \geq0})$ be a perfectoid tower arising from some pair $(R_0, I_0)$. 
Let $(\{ R_i^{s.\flat} \}_{i \geq 0}, \{ t_i^{s.\flat} \}_{i \geq 0})$ be the tilt.
Then the following assertions hold.
\begin{enumerate}
    \item 
    $(\{ R_i^{s.\flat} \}_{i \geq 0}, \{ t_i^{s.\flat} \}_{i \geq 0})$ is a perfect tower.
    In particular, $R_i^{s.\flat}$ is reduced.
    \item $I_0^{s.\flat}$ is a principal ideal.
    
    \item The $I_0$-torsions $(R_i)_{I_0\textnormal{-tor}}$ of $R_i$ is isomorphic to the $I_0^{s.\flat}$-torsions $(R_i^{s.\flat})_{I_0^{s.\flat}\textnormal{-tor}}$ of $R_i^{s.\flat}$ as non-unital rings.
    In particular, a generator $f_0$ of $I_0$ is a non-zero-divisor on $R_0$ if and only if so is a generator $f_0^{s.\flat}$ of $I_0^{s.\flat}$ on $R_0^{s.\flat}$ for any $i \geq 0$.
    
    \item The $I_0$-adic completion of the colimit $R_\infty :=\colim R_i$ is a perfectoid ring.
    Moreover, if $\widehat{R_\infty}$ denotes the $I_0$-adic completion of $R_\infty$, then the tilt $\widehat{R_\infty}^{\flat}$ is isomorphic to the $I_0^{s.\flat}$-adic completion of the colimit $R_\infty^{s.\flat} := \colim R_i^{s.\flat}$.
    
    \item For any $i \geq 0$, if $R_i$ is Noetherian, then so is $R_i^{s.\flat}$.
    Moreover, the converse holds if $R_i$ is $I_0$-adically complete and separated.
\end{enumerate}
\end{lemma}

\begin{proof}
    The assertion (1) is \cite[Proposition 3.10 (2)]{INS23}. 
    The assertions (2) and (3) follow from \cite[Theorem 3.35]{INS23}.
    The assertion (4) is \cite[Corollary 3.52]{INS23} and \cite[Lemma 3.55]{INS23}.
    The assertion (5) is \cite[Proposition 3.42 (2)]{INS23}.
\end{proof}

\begin{lemma}
    Let $( \{ R_i \}_{ i \geq 0}, \{ t_i \}_{i \geq0})$ be a perfectoid tower arising from some pair $(R_0, I_0)$.
    Let $(\{ R_i^{s.\flat} \}_{i \geq 0}, \{ t_i^{s.\flat} \}_{i \geq 0})$ be the tilt.
    Then for any $i \geq 0$, $R_i$ is local if and only if so is $R_i^{s.\flat}$ \footnote{The if part is already proved in \cite[Lemma 3.11 (2)]{INS23}.}.
\end{lemma}

\begin{proof}
    Since $R_i$ (resp. $R_i^{s.\flat}$) is $I_0$-adically Zariskian (resp. $I_0^{s.\flat}$-adically Zariskian), we know that $R_i$ (resp. $R_i^{s.\flat}$)is local if and only if so is $R_i/I_0R_i$ (resp. $R_i^{s.\flat}/I_0^{s.\flat}R_0^{s.\flat}$).
    Hence, by the isomorphism (\ref{PerfdTiltIsom}), we obtain the desired equivalence.
\end{proof}

\begin{lemma}
\label{RiTorsion}
Let $( \{ R_i \}_{ i \geq 0}, \{ t_i \}_{i \geq0})$ be a perfectoid tower arising from some pair $(R_0, I_0)$. Then the following assertions are equivalent.
\begin{enumerate}
\item $R_0$ is $I_0$-torsion-free.
\item $R_i$ is $I_0$-torsion-free for any $i \geq 0$.
\end{enumerate}
\end{lemma}

\begin{proof}
This immediately follows from the bijectivity of $(F_i)_{\textnormal{tor}}$ appearing in the axiom (g) of \Cref{DefPerfectoidTowers}.
\end{proof}

\begin{lemma}\label{FinitePerfdTower}
Let $(\{ R_i \}_{i \geq0 },\{t_i\}_{i \geq 0})$ be a perfectoid tower arising from some pair $(R,I_0)$.
Suppose that $0$-th small tilt $R_0^{s.\flat}$ is $F$-finite and each $R_i$ is $I_0$-adically complete and separated.
Then $t_i$ and $t_i^{s.\flat}$ is finite for each $i \geq 0$.
In particular, each $R_i$ (resp. $R_i^{s.\flat}$) is a finitely generated $R_0$-module ($R_i^{s.\flat}$-module).
\end{lemma}

\begin{proof}
Since the tilted tower is perfect, each map $t_i^{s.\flat}$ is finite if and only if $R_0^{s.\flat}$ is $F$-finite.
Moreover, by applying the topological Nakayama's lemma \cite[Theorem 8.4]{Mat86}, we obtain the finiteness of $t_i$.
\end{proof}

\section{$\delta$-rings, perfectoid towers, and lim Cohen--Macaulay sequences}
In this section, we discuss relations of lim Cohen--Macaulay sequences with $\delta$-rings and with perfectoid towers.
In \S \ref{SubSecTowerDelta}, we review some basics on $\delta$-rings and discuss splitting of Frobenius lifts.
In \S \ref{SubSecFrobPower}, we introduce Frobenius powers of ideals in $\delta$-rings and $\phi$-stability of ideals.
In \S \ref{SubSecConstlimCM}, we prove Main Theorem \ref{Mtheorem2}.
We also discuss the construction of lim Cohen--Macaulay sequences and big Cohen--Macaulay algebras from $\delta$-rings.

\subsection{Towers of $\delta$-rings}\label{SubSecTowerDelta}
Let $p$ be a prime number and let $A$ be a commutative ring. For the following material, see also \cite{Rezk}.

\begin{definition}
A pair $(A,\delta)$ is called a \textit{$\delta$-ring} if there is a function $\delta:A \to A$ such that $\delta(0)=\delta(1)=0$ and the following identities are satisfied.
\begin{enumerate}
\item
For $x,y \in A$, $\delta(xy)=x^p\delta(y)+y^p\delta(x)+p\delta(x)\delta(y)$.

\item
For $x,y \in A$, 
$$
\delta(x+y)=\delta(x)+\delta(y)+\dfrac{x^p+y^p-(x+y)^p}{p}.
$$
\end{enumerate}
This map $\delta$ is called a \textit{$p$-derivation on $A$}.
\end{definition}

For a $\delta$-ring $(A,\delta)$, the map $\phi(x)\coloneqq x^p+p\delta(x)$ is a lift of the Frobenius map on $A/pA$. 
We call it \textit{the associated Frobenius lift on $(A,\delta)$} (or \textit{a Frobenius lift on $A$} if there is no confusion).

Let $W_2(A)$ be the ring of Witt vectors of $A$ of length $2$ and $\pi:W_2(A) \to A$ be the projection into the first factor. Then giving a $\delta$-ring structure on $A$ is equivalent to giving a ring map $\omega:A \to W_2(A)$ such that $\pi \circ \omega=\id_A$. Indeed, write $\omega(x)=(x,\delta(x))$. Then the map $\delta:A \to A$ defines the $\delta$-ring structure on $A$ (see \cite[Remark 2.4]{BS22}). Let us recall the construction of towers from $\delta$-rings as introduced in \cite{I23}.

\begin{definition}
\label{pdelta0}
Let $(A,\delta)$ be a $\delta$-ring with the associated Frobenius lift $\phi$. We define a finite colimit
\begin{equation}
\label{perfectdelta0}
A_i\coloneqq \colim\big\{A\underbrace{\xrightarrow{\phi} A \xrightarrow{\phi} \cdots \xrightarrow{\phi}}_{i\textnormal{\ times}}A\big\}.
\end{equation}
We define the transition map $\tau_{i}: A_i \to A_{i+1}$ as the map of the colimit induced by the commutative diagram:
\begin{equation}
\label{perfectdelta1}
\vcenter{\xymatrix{
A \ar[d]^{\id} \ar[r]^{\phi}&A\ar[d]^{\id} \ar[r]^{\phi}& \cdots \ar[r]^{\phi} & A \ar[d]^{\id} \ar[rd]^{\phi}  \\
A \ar[r]_{\phi} & A\ar[r]_{\phi}& \cdots \ar[r]_{\phi} & A \ar[r]_{\phi} & A
}}
\end{equation}

The above datum gives rise to a tower of rings $(\{A_i\}_{i \geq 0},\{\tau_i\}_{i \geq 0})$.
By considering the similar diagram:
\begin{equation}
\label{perfectdelta2}
\vcenter{\xymatrix{
A/pA \ar[d]^{\id} \ar[r]^{\Frob}&A/pA\ar[d]^{\id} \ar[r]^{\Frob}& \cdots \ar[r]^{\Frob} & A/pA \ar[d]^{\id} \ar[rd]^{\Frob}  \\
A/pA \ar[r]_{\Frob} & A/pA\ar[r]_{\Frob}& \cdots \ar[r]_{\Frob} & A/pA \ar[r]_{\Frob} & A/pA
}}
\end{equation}
we have the tower of $\mathbb{F}_p$-algebras $(\{(A/pA)_i\}_{i \ge 0},\{\overline{\tau}_i\}_{i \ge 0})$.

Also, we define the ring map $\phi_i :  A_{i+1} \to A_i$ as the map induced by the following diagram:
\begin{equation}
\label{perfectdelta4}
\vcenter{\xymatrix{
A \ar[r]^{\phi} \ar[d]^{\phi} & A \ar[r]^{\phi} \ar[d]^{\phi} & \cdots \ar[r]^{\phi}& A \ar[r]^{\phi} \ar[d]^{\phi} & A \ar[ld]^{\id}\\
A \ar[r]^{\phi} & A \ar[r]^{\phi} & \cdots \ar[r]^{\phi} & A. 
}}
\end{equation}
\end{definition}

We continue the discussions regarding \Cref{pdelta0}. Let us assume that $A/pA$ is reduced (equivalently, the Frobenius map on $A/pA$ is injective). The transition map $\overline{\tau}_i$ is purely inseparable and the Frobenius map on $(A/pA)_i$ factors as $(A/pA)_i \twoheadrightarrow (A/pA)_{i-1} \xrightarrow{\overline{\tau}_i} (A/pA)_i$. In particular, this shows that $(A/pA)_i \cong (A/pA)^{1/p^i}$ as rings. Thus, we have the following fact.

\begin{lemma}
\label{purelyinseparable}
Assume that $A/pA$ is reduced. Then $(\{(A/pA)_i\}_{i \ge 0},\{\overline{\tau}_i\}_{i \ge 0})$ is a perfect tower.
\end{lemma}

The diagram $(\ref{perfectdelta1})$ is compatible with $(\ref{perfectdelta2})$ with respect to the quotient map $A \to A/pA$.
In other words, we obtain the commutative diagram of towers whose vertical maps are exact:
\begin{equation}
\label{perfectdelta3}
\vcenter{\xymatrix{
A=:A_0 \ar[d]^{\times p} \ar[r] &A_1\ar[d]^{\times p} \ar[r] & \cdots \ar[r] & A_n \ar[d]^{\times p} \ar[r] &\cdots \\
A=:A_0 \ar[d]^{\pi_0} \ar[r] &A_1\ar[d]^{\pi_1} \ar[r] & \cdots \ar[r] & A_n \ar[d]^{\pi_n} \ar[r] &\cdots \\
A/pA=:(A/pA)_0 \ar[r] & (A/pA)_1 \ar[r] & \cdots \ar[r] & (A/pA)_n \ar[r] & \cdots \\
}}
\end{equation}
Note that the maps $\pi_i$ $(i=0,1,2,\ldots)$ are surjective with kernel generated by $p$. 

\begin{lemma}\label{FrobLiftInjpZariskian}
    Let $(A,\delta)$ be a $p$-torsion-free $\delta$-ring with the associated Frobenius lift $\phi$.
    Suppose that $A$ is a $p$-Zariskian Noetherian ring such that $A/pA$ is reduced.
    Then $\phi$ is injective.
\end{lemma}

\begin{proof}
    Since $A$ is $p$-adically separated by applying \cite{FGK11}, this follows from \cite[Lemma 3.7]{I23}.
\end{proof}

\begin{remark}
We remark that the paper \cite{IN24} exploits the Frobenius structure on the prismatic cohomology to prove a Kunz-type regularity theorem in mixed characteristic.
\end{remark}

\subsection{Frobenius powers of ideals in $\delta$-rings}\label{SubSecFrobPower}

In commutative ring theory in positive characteristic, the Frobenius powers of ideals of rings are one of the main tools to analyze properties of rings and ideals.
In this subsection, let us consider the properties of Frobenius powers of ideals in $\delta$-rings.

We define the $\phi$-powers of ideals.

\begin{definition}
Let $(A,\delta)$ be a $\delta$-ring with the associated Frobenius lift $\phi$ and let $I$ be an ideal of $A$.
Then $I^{[\phi]}$ is defined as the ideal of $A$ generated by $\phi(a)$ for all $a \in I$. 
\end{definition}

One can define the $\phi$-powers of $I$ even if $I$ is not finitely generated.
However, if $I$ is finitely generated, then $I^{[\phi]}$ is independent of the choice of generators.

\begin{lemma}
\label{FrobeniusLiftPower}
Let $(A, \delta)$ be a $\delta$-ring with the associated Frobenius lift $\phi$ and let $I$ be a finitely generated ideal of $A$ generated by $f_1,\ldots, f_r$.
Then $I^{[\phi]}$ is a finitely generated ideal generated by $\phi(f_1),\ldots, \phi(f_r)$.
\end{lemma}

\begin{proof}
This is straightforward.
\end{proof}

If $A$ is an $\mathbb{F}_p$-algebra, then the Frobenius power of an ideal $I^{[p]}$ is contained by $I$.
However, $I^{[\phi]}$ is not in general contained in $I$.

\begin{example}
Set $A:=\mathbb{Z}_p\llbracket x,y \rrbracket$ and $I := (p^3+x^3+y^3)$.
Then $I$ does not contain $I^{[\phi]}$.
Indeed, by Lemma \ref{FrobeniusLiftPower}, we obtain that $I^{[\phi]} = (p^3+x^{3p}+y^{3p})$.
Suppose that $p^3+x^{3p} + y^{3p} = g\cdot(p^3+x^3+y^3)$ for some $g \in A$.
We have $(1-g)p^3=x^{3}(x^{3p-3}-g)+y^3(y^{3p-3}-g)$, which means that $p$ divides $x^{3p-3}-g$ and $y^{3p-3}-g$.
This implies that $x^{3p-3}$ and $y^{3p-3}$ is also divided by $p$.
But this is a contradiction.
\end{example}

\begin{definition}
Let $(A, \delta)$ be a $\delta$-ring with the associated Frobenius lift $\phi$. Then an ideal $I\subseteq A$ is \textit{$\phi$-stable} if the inclusion $I^{[\phi]} \subseteq I$ holds.
\end{definition}

By the comments on Ishizuka in \cite[Proposition 6.9]{I23}, if $A/I$ is $p$-torsion-free for a $\delta$-ring $(A,\delta)$, $I$ is $\phi$-stable if and only if $I$ is $\delta$-stable in the sense of $\delta(I) \subset I$.
In this case, $A/I$ is equipped with the $\delta$-structure induced by $(A,\delta)$.
However, if $A/I$ is not $p$-torsion-free, $A/I$ is not necessarily equipped with the $\delta$-structure even if $I$ is $\phi$-stable.
The following example is found during the first author's discussion with Ishizuka.

\begin{example}\label{egnon-deltaring}
    Set $A := \mathbb{Z}_p \llbracket x \rrbracket$ equipped with the $p$-derivation $\delta$ such that $\delta(x) =0$ and $I := (px)$.
    Then $R := \mathbb{Z}_p \llbracket x \rrbracket / (px)$ admits the Frobenius lift induced by that of $A$ because $I$ is a $\phi$-stable ideal in $A$.
    Nevertheless, $R$ is not equipped with a $\delta$-structure.
    Indeed, assume that there exists a $p$-derivation $\delta : R \to R$ on $R$.
    Then the following equations hold in $R$.
    $$
    0
    = \delta(px) 
    = p^p\delta(x) + x^p\delta(p) + p\delta(p)\delta(x) 
    = p\delta(x)(p^{p-1}+\delta(p))+x^p\delta(p).
    $$
    Since $\delta(p)$ is a unit in $R$, this implies that $x^p \in (p)$ in $A$, but this is a contradiction.
\end{example}

It is easy to see that an ideal $I$ of a $\delta$-ring $(A,\delta)$ with the associated Frobenius lift $\phi$ is $\phi$-stable if and only if $\phi(x)$ lies in $I$ for any $x \in I$.
Also, several operations on ideals are compatible with Frobenius lifts. The proof of the following lemma is immediate.


\begin{lemma}
\label{IdealOperationsStable}
Let $(A,\delta)$ be a $\delta$-ring with the associated Frobenius lift $\phi$. 
Let $I,J$ be $\phi$-stable ideals of $A$.
Then the following ideals are also $\phi$-stable:
\begin{enumerate}
    \item The intersection $I \cap J$.
    \item The sum $I + J$.
    \item The product $I J$.
    \item The radical $\sqrt{I}$.
\end{enumerate}
\end{lemma}

The above lemma provides examples of $\phi$-stable ideals.

\begin{example}\label{ExamplePhiStables}
Set $A := W(k)\llbracket x_2,\ldots, x_d \rrbracket$ equipped with a $p$-derivation $\delta$ defined by $\delta(x_i) = 0$ for any $2 \le i \le d$.
In this case, we have $\phi(x_i) = x_i^p$ for any $2 \leq i \le d$, where $\phi$ is the Frobenius lift defined by $\delta$.
\begin{enumerate}
\item 
An arbitrary monomial ideal in $A$ is $\phi$-stable.
Indeed, let $I = (M_1,\ldots, M_r)$ be a monomial ideal in $x_1 := p, x_2, \ldots, x_d$. By an easy calculation, it follows that $\phi(M) \in (M)$ for any monomial $M = x_1^{l_1} x_2^{l_2} \cdots x_d^{l_d}$ in $A$. Hence, by \Cref{FrobeniusLiftPower}, it follows that $I^{[\phi]} = (\phi(M_1), \ldots, \phi(M_r)) \subseteq (M_1, \ldots, M_r)$.

\item
The binomial ideal generated by $p^n(M_1- M_2)$, where $M_1$ and $M_2$ are monomials in $p,x_2,\ldots,x_d$ that are not divided by $p$ in $A$, is $\phi$-stable. Indeed, it suffices to show that $p^n(x_2^{pk_1}\cdots x_n^{pk_n} - x_2^{pl_1}\cdots x_n^{pl_n}) \in\left(p^n(x_2^{k_1}\cdots x_n^{k_n} - x_2^{l_1}\cdots x_n^{l_n})\right)$. But we have 
$$
p^n(x_2^{pk_2}\cdots x_n^{pk_n} - x_2^{pl_1}\cdots x_n^{pl_n}) = p^n(x_2^{k_2}\cdots x_n^{k_n} - x_2^{l_2}\cdots x_n^{l_n}) g \in \left(p^n(x_2^{k_1}\cdots x_n^{k_n} - x_2^{l_1}\cdots x_n^{l_n})\right)
$$
for some $g \in W(k)\llbracket x_2,\ldots,x_n  \rrbracket$.
\end{enumerate}
\end{example}

\begin{lemma}
\label{phiInjection}
Let $(A,\delta)$ be a $\delta$-ring with the associated Frobenius lift $\phi$.
Let $I$ be a $\phi$-stable ideal. Suppose that $R := A/I$ is reduced and $\phi$ is faithfully flat. Then the ring homomorphism $\phi_R : R \to R$ induced by $\phi$ is injective.
\end{lemma}

\begin{proof}
Pick an element $\overline{f} \in \Ker \phi_R$ with a lift $f \in A$.
By the assumption, it suffices to show $f^p \in I$.
Since an element $\overline{p\delta(f)}$ lies in $\Ker \phi_R$, an element $\phi_A(p\delta(f))$ is in $I$. By combining it with the faithfully flatness of $\phi : A \to A$, we obtain $p\delta(f) \in I(\phi_*A) \cap A = I$. Hence $f^p+p\delta(f) \in I$ implies $f^p \in I$, as desired.
\end{proof}

\begin{remark}
\begin{enumerate}
    \item In the setting of \Cref{phiInjection}, the class of complete Noetherian local $\delta$-rings whose Frobenius lifts are faithfully flat is exactly the class of regular local rings. This is the result of Ishizuka and Nakazato in \cite{IN24}.
    \item Yoshikawa proves that a complete regular local ring which admits a Frobenius lift must be unramified.
    For details, we refer the reader to \cite[Proposition 2.1]{Yos25}.
\end{enumerate}
\end{remark}

\subsection{Constructions of lim Cohen--Macaulay sequences}\label{SubSecConstlimCM}

The existence of $\delta$-ring structure on Noetherian local rings of mixed characteristic is related to the existence of lim Cohen--Macaulay sequences, which are the notion introduced by Bhatt, Hochster, and Ma in \cite{BHM24}.
See \cite{Ho17} and \cite{Ma23} for interesting applications.

\begin{definition}
\label{LimCMDef}
Let $(R,\fm,k)$ be a Noetherian local ring with $d\coloneqq \dim R$. Then a sequence of finitely generated $R$-modules $\{M_n\}_{n \ge 0}$ is called a \textit{lim Cohen--Macaulay sequence} if $\dim  M_n=d$ for all $n \ge 0$ and for all $i \ge 1$,
$$
\ell_R(H_i(\underline{x};M_n))=\omicron\big(\dim_k(k \otimes_R M_n)\big)
$$
for some system of parameters $\underline{x}\coloneqq x_1,\ldots,x_d$ of $R$. Note that the number $\dim_k(k \otimes_R M_n)$ is the same as the number of minimal generators of $M_n$ over $R$ by Nakayama's lemma.
\end{definition}

Note that we do not assume $\{M_n\}_{n \geq 0}$ in \Cref{LimCMDef} to be a directed (or inverse) system of $R$-modules. The next proposition is a special case of \cite[Proposition 4.12]{BHM24}.

\begin{proposition}
\label{LimCMSeqProp}
Let $(R,\fm,k)$ be a $d$-dimensional Noetherian local ring and let $\{M_n\}_{n \ge 0}$ be a sequence of finitely generated $R$-modules satisfying $d=\dim M_n$ for all $n \ge 0$. Suppose that $x \in \fm$ is a parameter element of $R$ such that $x$ is a regular element of $M_n$ for all $n \ge 0$. Then $\{M_n\}_{n \ge 0}$ is a lim Cohen--Macaulay sequence over $R$ if and only if $\{M_n/xM_n\}_{n \ge 0}$ is a lim Cohen--Macaulay sequence over $R/xR$.
\end{proposition}

\begin{proof}
By assumption, we have $\dim R/xR=\dim M_n/xM_n=d-1$, and we can extend $x \in \fm$ to a full system of parameters $\underline{x}\coloneqq x,x_2,\ldots,x_d$ of $R$. 
It follows from \cite[Corollary 1.6.13]{BH98} that
$$
H_i(\underline{x};M_n) \cong H_i(x_2,\ldots,x_d;M_n/xM_n)
$$
for all $i \ge 0$ and $n \ge 0$. So we have
$$
\ell_R(H_i(\underline{x};M_n))=\ell_{R/xR}(H_i(x_2,\ldots,x_d;M_n/xM_n)).
$$
Since $x \in \fm$, we get $k \otimes_R M_n \cong k \otimes_{R/xR} M_n/xM_n$ and 
$$
\dim_k(k \otimes_R M_n)=\dim_k(k \otimes_{R/xR} M_n/xM_n).
$$
The proof is now obtained.
\end{proof}

We prove the first main theorem in this paper.

\begin{theorem}
\label{weakperfectlimCM}
Let $(R,\fm,k)$ be a complete Noetherian local domain of mixed characteristic with $F$-finite residue field of characteristic $p>0$. 
Suppose that there exists a perfectoid tower $( \{ R_i \}_{i \geq 0}, \{t_i\}_{i \geq 0})$ arising from $(R, I_0)$ for some ideal $I_0 \subseteq R$. 
Then $\{ R_n \}_{n \geq 0} $ is a lim Cohen--Macaulay sequence of algebras.
\end{theorem}

\begin{proof}
We know that each $R_i$ (resp. $R_i^{s.\flat}$) is a finitely generated $R_0$-module (resp. $R_0^{s.\flat}$-module) by \Cref{FinitePerfdTower}.
Fix a generator $f_0$ (resp. $f_0^{s.\flat}$) of $I_0$ (resp. $I_0^{s.\flat}$).
Then $\{ R_n \}_{n \geq 0}$ (resp. $\{ R_n^{s.\flat} \}_{n \geq 0}$) is a lim Cohen--Macaulay sequence if and only if so is $\{ R_n/f_0R_n \}_{n \geq0}$ (resp. $\{ R_n^{s.\flat}/ f_0^{s.\flat} R_n^{s.\flat}\}_{n \geq 0}$) by \Cref{RiTorsion} and \Cref{LimCMSeqProp}.
Since $R_i^{s.\flat}/f_0^{s.\flat} R_i^{s.\flat}$ is isomorphic to  $R_i/f_0 R_i$ for any $i \geq 0$, we conclude that $\{ R_n \}_{n \geq 0}$ is lim Cohen--Macaulay if and only if so is $\{ R_n^{s.\flat} \}_{n \geq 0}$. 
Hence it suffices to show that $\{ R_n^{s.\flat} \}_{n \geq 0}$ is lim Cohen--Macaulay.
Now, a perfect tower arising from an $F$-finite reduced Noetherian local ring is a lim Cohen--Macaulay sequence by \cite[Theorem 5.4 and Remark 5.5]{BHM24}.
This implies that $\{ R_n^{s.\flat} \}_{n\geq 0}$ is lim Cohen--Macaulay, as desired.
\end{proof}

In the rest of this section, we give two results, which are the existence of a lim Cohen--Macaulay sequence and a big Cohen--Macaulay algebra over a $\delta$-ring.
The first result provides an example of a lim Cohen--Macaulay sequence but not a perfectoid tower.
Indeed, the tower defined by the commutative diagram (\ref{perfectdelta1}) is purely inseparable (\Cref{ConstPurelyInsep}), and the associated sequence is lim Cohen--Macaulay (\Cref{LimCMSeq}), but not a perfectoid tower (\Cref{RemarkPurelyInseplimCM}).

\begin{corollary}
\label{LimCMSeq}
Let $(R,\fm,k)$ be a $p$-torsion-free\footnote{For any $\delta$-ring $(R,\delta)$, if the associated Frobenius lift $\phi$ on $R$ is injective or $R$ is reduced, then $R$ is $p$-torsion-free (see \cite[Lemma 2.28]{BS22}).} Noetherian complete local ring of mixed characteristic $p>0$ with perfect residue field. 
Assume that $R$ admits a $\delta$-ring structure given by a $p$-derivation $\delta:R \to R$ such that $R/pR$ is reduced.
Let $(\{R_i\}_{i \geq 0},\{\tau_i\}_{ i \geq 0})$ be the tower, where each $\tau_i$ is defined by the diagram $(\ref{perfectdelta1})$. 
Then the sequence of $R$-algebras $\{R_n\}_{n \geq 0}$ is a lim Cohen--Macaulay sequence of algebras.
\end{corollary}

\begin{proof}
Under the stated assumptions, $R/pR$ is $F$-finite and $R/pR \hookrightarrow (R/pR)^{1/p^i}$ is a module-finite extension for each $i \geq 0$. 
Since $p$ is a nonzero divisor of $R$ and $(R/pR)^{1/p^i} \cong R_i/pR_i$, the Krull dimension of $R_i$ is $\dim R$.
Fix a system of parameters $\underline{x}\coloneqq p,x_2,\ldots,x_d$ for $R$. 
Then as in $(\ref{perfectdelta3})$, we have a short exact sequence of finitely generated $R$-modules: $0 \to R_n \xrightarrow{\times p} R_n \to (R/pR)_n \to 0$. 
Moreover by \Cref{purelyinseparable}, $\{ (R/pR)_n \}_{n \geq 0}$ is a lim Cohen--Macaulay sequence in view of \cite[Theorem 5.4]{BHM24} or \cite[Theorem 5.1]{Ho17}. It follows from \Cref{LimCMSeqProp} that $\{R_n \}_{n \ge 0}$ is a lim Cohen--Macaulay sequence.
\end{proof}

According to the following proposition, it can be proved that any Noetherian local domain admitting a $\delta$-ring structure has a big Cohen--Macaulay algebra with some good properties. The proof does not depend on Andr\'e's theorem (see \cite{An1} and \cite{An2} for his results).

\begin{proposition}
Let $(R,\fm,k)$ be a Noetherian local domain of mixed characteristic $(0, p)$.
Assume that $R$ admits a $\delta$-ring structure (equivalently, structure of a crystalline prism).
Then there exists a $p$-adically complete $R$-algebra $C$ such that $C$ is an $\fm$-adically complete balanced big Cohen--Macaulay $R$-domain, and $C/pC$ is an absolutely integrally closed domain.
\end{proposition}

\begin{proof}
Let $\phi:R \to R$ be the Frobenius lift attached to $\delta$ as claimed in the assumption. Set $R_\infty$ to be the $p$-adic completion of the colimit $\varinjlim_{\phi}R$. Then $(R_\infty,pR_\infty)$ is a perfect prism and we have $R_\infty \cong W(R_\infty/pR_\infty)$, where $W(-)$ is the functor of $p$-typical Witt vectors. By \cite[Theorem 7.8]{Di07} and \cite[Theorem 17.5.96]{GR23}, there exist an absolutely integrally closed $\mathbb{F}_p$-domain $B$ and a ring map $R_\infty/pR_\infty \to B$ such that $B$ is complete with respect to the adic topology generated by a system of parameters of $R/pR$ and $B$ is a balanced big Cohen--Macaulay $R/pR$-algebra. Since $R_\infty/pR_\infty \to B$ is a ring map of perfect $\mathbb{F}_p$-algebras, it lifts to a ring map $R_\infty \to C\coloneqq W(B)$. Let $p,x_2,\ldots,x_d$ be a system of parameters of $R$. Then $C/pC$ is $(x_2,\ldots,x_d)$-adically complete by construction. So we can apply \cite[Theorem 2.9]{BMP+23} and \cite[Lemma 4.2]{Sh15} to conclude that $C$ has the desired properties.
\end{proof}

\section{Deformation of perfectoid purity via perfectoid towers}
In this section, we discuss the deformation of perfectoid purity in view of the existence of perfectoid towers.
First we recall the definition of perfectoid pure and perfectoid injective singularities.

\begin{definition}[{\cite[Definition 4.1]{BMPSTWW24}}]\label{def:perfectoidpureinj}
    Let $R$ be a Noetherian ring with $p$ contained in its Jacobson radical.
    \begin{enumerate}
    \item
    $R$ is called \textit{perfectoid pure} if there exists a perfectoid $R$-algebra $B$ such that $R \to B$ is pure.
    \item 
    $R$ is called \textit{perfectoid injective} if there is a perfectoid $R$-algebra $B$ such that the local cohomology map $H^i_\fm(R) \to H^i_\fm(B)$ is injective for all $i \ge 0$ and all maximal ideals $\fm$ of $R$.
    \end{enumerate}
\end{definition}

\begin{remark}
    In general, for a Noetherian ring $R$ with $p$ contained in its Jacobson radical, if $R$ is perfectoid pure, then it is perfectoid injective.
    In addition, if $R$ is quasi-Gorenstein, then the converse implication holds.
    For details, see \cite[Lemma 4.4]{BMPSTWW24}.
    They also introduced lim-perfectoid purity and lim-perfectoid injectivity by using the perfectoidization defined by Bhatt and Scholze \cite{BS22}.
    In general, perfectoid pure (resp. perfectoid injective) singularities are lim-perfectoid pure (resp. lim-perfectoid injective), and all four definitions are equivalent if $R$ is locally complete intersection.
\end{remark}

\begin{example}
    By \cite[Lemma 4.6]{BMPSTWW24}, a pure subring of a perfectoid pure ring is also perfectoid pure.
    Using this result, one can show that any local log-regular ring is perfectoid pure.
    Indeed, we may assume that a local log-regular ring is complete with respect to its maximal ideal by \cite[Lemma 4.8]{BMPSTWW24}.
    By the discussion in the proof of \cite[Proposition 4.8]{Ish25}, we know that any complete local log-regular ring is a pure subring of a complete regular local ring.
    Since a regular local ring is perfectoid pure, we obtain that a complete local log-regular ring is perfectoid pure.
\end{example}

The following lemma is useful in proving that rings admitting perfectoid towers are perfectoid pure.
\begin{lemma}\label{PureCompletion}
Let $(R,\fm,k)$ be a Noetherian local ring and let $I$ be an ideal of $R$.
Let $S$ be an $R$-algebra.
Suppose that $R/I^n \to S/I^nS$ is pure for any $n > 0$.
Then $R \to \widehat{S}$ is pure where $\widehat{S}$ is the $I$-adic completion of $S$.
\end{lemma}

\begin{proof}
Let $E(k)$ be an injective hull of the residue field $k$.
It suffices to show that $E(k) \to \widehat{S}\otimes_RE(k)$ is injective. Moreover, since $E(k)$ is isomorphic to the colimit of finitely generated sub $R$-modules, it suffices to show that the map $N \to \widehat{S} \otimes_R N$ is pure for every finitely generated submodule $N \subset E(k)$.
Since $E(k)$ is $\fm$-power torsion, there is an integer $n>0$ such that $\fm^n N =0$.
In particular $I^nN =0$ because $R$ is local.
Hence $N \to \widehat{S}\otimes_R N $ can be identified with $N \to S/I^nS \otimes_{R/I^n}N$.
By assumption of the purity of $R/I^n \to S/I^nS$, $N \to \widehat{S}\otimes_R N$ is injective.
\end{proof}

Perfectoid purity is defined by the existence of a pure ring map from a given Noetherian ring to a perfectoid ring.
The target of such a pure map is non-Noetherian because perfectoid rings are usually non-Noetherian.
However, if a ring admits a perfectoid tower, one can reduce to checking the purity of the maps to the finite layers, which may be Noetherian.

\begin{proposition}
\label{PerfdPureTower}
Let $(R,\fm,k)$ be a Noetherian local ring and let $I_0$ be an ideal of $R$.
Let $(\{ R_i \}_{i \geq 0}, \{ t_i \}_{i \geq 0})$ be a perfectoid tower arising from $(R,I_0)$.
Suppose that $t_{0,i} := t_{i-1}\circ \cdots \circ t_0 : R_0 \to R_i$ is pure for any $i \geq 0$.
Then $R$ is perfectoid pure.
\end{proposition}

\begin{proof}
    By assumption, we have a composition $R_0 \to \widehat{R_\infty}$ of a pure ring map $R_0 \to R_\infty$ and the $I_0$-adic completion $R_\infty \to \widehat{R_\infty}$.
    Note that $\widehat{R_\infty}$ is a perfectoid ring by \cref{PropPerfdTower} (4).
    Also, $R_0/I_0^n \to \widehat{R_\infty}/I_0^n\widehat{R_\infty} \cong R_\infty/I_0^n R_\infty$ is pure where $\widehat{R_\infty}$ is the $I_0$-adic completion of $R_\infty$.
    Hence, by \Cref{PureCompletion}, $R_0 \to \widehat{R_\infty}$ is pure.
    This implies that $R$ is perfectoid pure.
\end{proof}

We prove the key lemma to prove the deformation of perfectoid purity in the case that a ring admits a perfectoid tower.
\begin{lemma}
\label{deltasplitv2}
Let $A$ be a Gorenstein ring and let $f:A \hookrightarrow B$ be a module-finite extension where $B$ is an $A$-algebra.
Let $x \in A$ be a non-zero-divisor on $B$.
Suppose that $B/xB$ is a maximal Cohen--Macaulay $A/xA$-module and $\overline{f} : A/xA \to B/xB$ splits as $A/xA$-module.
Then $f$ also splits.
\end{lemma}

\begin{proof}
It suffices to show that $\Ext^{1}_A(C,A)=0$.
Consider a commutative diagram

\[
\xymatrix@M=10pt{ 
0 \ar[r]& A \ar[r]^{f} \ar[d]^{\times x} & B \ar[r] \ar[d]^{\times x} & C \ar[r] \ar[d]^{\times x}&0 \\
0 \ar[r] &A \ar[r]^{f} & B \ar[r] &C \ar[r] &0 \\
}
\]
Since $\overline{f}$ splits by assumption, we get a split exact sequence of $A/xA$-modules: $0 \to A/xA \xrightarrow{\overline{f}} B/xB \to C/xC \to 0$. Moreover, $C$ is $x$-torsion-free by the snake lemma. We claim that
\begin{equation}
\label{ExtVan}
\Ext^1_{A/xA}(C/xC,A/xA)=0. 
\end{equation}
To prove $(\ref{ExtVan})$, we may replace $A$ with the completion-localization $\widehat{A_\fm}$ without losing generality. 
Since we have an exact sequence
$$
0 \to \Ext^1_{A/xA}(C/xC,A/xA) \to \Ext^1_{A/xA}(B/xB,A/xA) \to \Ext^1_{A/xA}(A/xA,A/xA)=0,
$$
it is sufficient to show $\Ext^1_{A/xA}(B/xB,A/xA)=0$. Since $A/xA$ is Gorenstein, the canonical module of $A/xA$ is itself and $B/xB$ is a maximal Cohen--Macaulay $A/xA$-module. By local duality theorem \cite[Corollary 3.5.11]{BH98}, we obtain $(\ref{ExtVan})$. By applying a homological lemma of Rees \cite[Lemma 3.1.16]{BH98}, we get $\Ext^2_A(C/xC,A)=0$. The exact sequence: $0 \to C \xrightarrow{\times x} C \to C/xC \to 0$ yields an exact sequence:
$$
\cdots \to \Ext^1_A(C/xC,A) \to \Ext^1_A(C,A) \xrightarrow{\times x} \Ext^1_A(C,A) \to \Ext^2_A(C/xC,A)=0.
$$
Hence we must get $\Ext^1_A(C,A)=x\Ext^1_A(C,A)$. Also, $\Ext_A(C,A)$ is also finitely generated because $A$ is a Noetherian local ring and $C$ is finitely generated.
This implies $\Ext^1_A(C,A)=0$ by Nakayama's lemma, as desired.
\end{proof}

As one of the applications of \Cref{deltasplitv2}, we show that the Frobenius lift on a $\delta$-ring splits.

\begin{corollary}
\label{deltasplit}
Let $(A,\delta)$ be a $p$-torsion-free $\delta$-ring such that $A$ is a $p$-Zariskian Noetherian ring and $A/pA$ is $F$-finite, $F$-split and Gorenstein. Then the $A$-module map $\phi:A \to A$ splits.
\end{corollary}

\begin{proof}
By the same argument as in \Cref{deltasplitv2}, we obtain $\Ext_A^1(C,A)=p\Ext_A^2(C,A)$ where $C$ is the cokernel of the injective Frobenius lift $A \xrightarrow{\phi} A^{(1)}$.\footnote{The injectivity of $\phi$ follows from \Cref{FrobLiftInjpZariskian}} In our setting, Nakayama's lemma applies, and hence we conclude that $\Ext_A^1(C,A)=0$.
\end{proof}

\begin{remark}
The Gorenstein hypothesis seems to be necessary in the hypothesis of \Cref{deltasplit}. Indeed, the positive characteristic analog of the deformation problem for $F$-splitting fails, as demonstrated in \cite{Singh98}.
\end{remark}

We prove Main Theorem \ref{MainTheorem3}.

\begin{theorem}\label{DeformPerfdPurity}
    Let $(R,\fm,k)$ be a Gorenstein local ring whose residue field is $F$-finite.
    Let $(\{ R_i \}_{i \geq 0}, \{ t_i \}_{i \geq 0})$ be a perfectoid tower arising from some pair $(R, I_0)$ where $R$ is $I_0$-torsion-free.
    Suppose that $\overline{R} =R/I_0R$ is $F$-split. 
    Then the transition map $t_{0, i} : R_0 \to R_i$ splits for every $i \geq 0$. In particular, $R$ is perfectoid pure. 
\end{theorem}

\begin{proof}
    The perfectoid purity of $R$ follows from the combination of the first assertion and \Cref{PerfdPureTower}.
    Hence let us show the first assertion.
    Set $\overline{R_i^{s.\flat}} := R_i^{s.\flat}/I_0^{s.\flat} R_i^{s.\flat}$.
    Since $R_i^{s.\flat} $ is isomorphic to $(R_0^{s.\flat})^{1/p^i}$ and $R_0^{s,\flat}$ is also Gorenstein, $R_i^{s.\flat}$ is a maximal Cohen--Macaulay $R_0^{s.\flat}$-module.
    This implies that $\overline{R_i}$ is a maximal Cohen--Macaulay $\overline{R_0}$-module.
    By combining \Cref{deltasplitv2} with the assumption, we obtain that $t_0^{s.\flat} : R_0^{s.\flat} \to R_1^{s.\flat}$ splits, hence $R_0^{s.\flat}$ is $F$-split.
    Also, it follows that $t_{0,i}^{s.\flat} :=t_{i-1}^{s.\flat} \circ \cdots \circ t_0^{s.\flat} : R_0^{s.\flat} \to R_i^{s.\flat}$ also splits.
    By the base change, we obtain that $\overline{R_0} \to \overline{R_i}$ splits.
\end{proof}

\section{Constructions of perfectoid towers}
In this section, we provide examples of perfectoid towers arising from quotients by $\phi$-stable ideals.
In \S \ref{SubSecpFreePerfdTower}, we prove \Cref{MainTheorem2} (\Cref{pTorFreePerfdTowers}).
The resulting perfectoid towers agree with those constructed by Ishizuka \cite{I23}, although the proofs are conmpletely different.
In \S \ref{SubSectionDet} (resp. \S \ref{SubSecSecRing}), we give examples of perfectoid towers arising from determinantal rings defined by $2\times m$-matrices (resp. section rings of smooth projective varieties).
In \S \ref{SubSubSectMonomialPerfd}, we discuss a construction of perfectoid towers arising from quotients by squarefree monomial ideals that contain $p$-torsion elements.

First, we define the colimits using the ring homomorphisms induced by Frobenius lifts on $\delta$-rings.

\begin{definition}
\label{Rpdelta0}
Let $(A,\delta)$ be a $\delta$-ring with the associated Frobenius lift $\phi$, let $I$ be a $\phi$-stable ideal, and let $R$ be the quotient ring $A/I$. Then we define the ring $R_i \coloneqq \colim\big\{R\underbrace{\xrightarrow{\phi_R} R \xrightarrow{\phi_R} \cdots \xrightarrow{\phi_R}}_{i\textnormal{\ times}}R\big\}.$
\end{definition}

In the same way as in Definition \ref{pdelta0} (\ref{perfectdelta1}) and (\ref{perfectdelta4}), we define the transition map $(\tau_R)_{i}: R_i \to R_{i+1}$ as the map of the colimit induced by the commutative diagram:
\begin{equation}
\label{perfectdeltaR1}
\vcenter{\xymatrix{
R \ar[d]^{\id} \ar[r]^{\phi_R}& R\ar[d]^{\id} \ar[r]^{\phi_R}& \cdots \ar[r]^{\phi_R} & R \ar[d]^{\id} \ar[rd]^{\phi_R}  \\
R \ar[r]_{\phi_R} & R \ar[r]_{\phi_R}& \cdots \ar[r]_{\phi_R} & R \ar[r]_{\phi_R} & R
}}
\end{equation}

Hence, we obtain the tower
\begin{equation}\label{PerfectionFrob}
    R_0 = R \xrightarrow{(\tau_R)_0} R_1 \xrightarrow{(\tau_R)_1} R_2 \xrightarrow{(\tau_R)_2} \cdots.
\end{equation}

Moreover, we define the ring map $(\phi_R)_i :  R_{i+1} \to R_i$ as the map induced by the following diagram: 
\begin{equation}
\label{perfectdeltaR2}
\vcenter{\xymatrix{
R \ar[r]^{\phi_R} \ar[d]^{\phi_R} & R \ar[r]^{\phi_R} \ar[d]^{\phi_R} & \cdots \ar[r]^{\phi_R}& R \ar[r]^{\phi_R} \ar[d]^{\phi_R} & R \ar[ld]^{\id_R}\\
R \ar[r]^{\phi_R} & R \ar[r]^{\phi_R} & \cdots \ar[r]^{\phi_R} & R. 
}}
\end{equation}

By combining the diagrams (\ref{perfectdelta4}) with (\ref{perfectdeltaR2}), we obtain the commutative square
\begin{equation}
\label{ARcommutative}
\vcenter{\xymatrix{
A_{i+1} \ar[r]^{\phi_i} \ar@{->>}[d] & A_i \ar@{->>}[d] \\
R_{i+1} \ar[r]^{(\phi_R)_i} \ar[r] & R_i.
}}
\end{equation}

Set the composition of the maps
\begin{equation}
\label{CompositeMap12}
\phi_{i,0} \coloneqq \phi_0\circ \phi_1\circ \cdots\circ \phi_{i-1} : A_i \to A_0,
\end{equation}
where $\phi_i$ is defined in \Cref{pdelta0} (\ref{perfectdelta4}).
Then the kernel of the surjection $\pi_i : A_i \twoheadrightarrow R_i$, which is denoted by $I_i$, is described as follows.

\begin{lemma}
\label{IsomRi}
Keep the notation as above. Then $I_i$ is contained in $\phi_{i,0}^{-1}(I) = \pi^{-1}_i(\Ker((\phi_R)_{i,0}))$. 
In particular, $R_i$ is isomorphic to $A_i/\phi_{i,0}^{-1}(I)$ if $\phi_R : R \to R$ is injective.
\end{lemma}

\begin{proof}
It is obvious that $\Ker \pi_i$ is contained in $\pi^{-1}_i(\Ker((\phi_R)_{i,0}))$.
The composition of the diagram $(\ref{ARcommutative})$ induces the equalities $\phi^{-1}_{i,0}(I) = \pi^{-1}_i(\Ker((\phi_R)_{i,0}))$.
If $\phi_R$ is injective, so is $(\phi_R)_{i,0}$.
Hence the assertions hold.
\end{proof}

\begin{definition}
\label{Ipdelta0}
With the notation in \Cref{IsomRi}, we denote $\phi_{i,0}^{-1}(I)$ as defined in $(\ref{CompositeMap12})$ by $I_i$.
\end{definition}

Here we check that the tower (\ref{PerfectionFrob}) is purely inseparable if $R$ is reduced.

\begin{proposition}\label{ConstPurelyInsep}
Keep the notation as above.
Suppose that $\overline{R}\coloneqq R/pR$ is reduced. 
Then the tower $(\{ R_i \}_{i\geq 0}, \{(\tau_R)_i\}_{i \ge 0} )$ is a purely inseparable tower that arises from $(R, (p))$.
\end{proposition}

\begin{proof}
The axiom (a) is trivial, and the axiom (b) follows from the assumption that $\overline{R}$ is reduced.
The induced ring map $\overline{\phi_i} : \overline{R_{i+1}} \to \overline{R_i}$ is the $i$-th Frobenius projection of the perfect tower $\{ R_0 \xrightarrow{\overline{(\tau_R)_0}} R_1 \xrightarrow{\overline{(\tau_R)_1}} \cdots\}$.
This implies that axiom (c) holds.
\end{proof}

\begin{remark}\label{RemarkPurelyInseplimCM}
The tower appearing in \Cref{ConstPurelyInsep} may not be a perfectoid tower, because it does not necessarily satisfy the axiom (f).
In order for a purely inseparable tower appearing in \Cref{ConstPurelyInsep} to become a perfectoid tower, it is necessary to suppose that $(A,J)$ is a prism defined by Bhatt-Scholze in \cite{BS22}.
A construction of perfectoid towers arising from prisms is provided by Ishizuka in \cite{I23}.
\end{remark}

\subsection{A construction of perfectoid towers of $R= A/I$ ($p$-torsion-free case)}\label{SubSecpFreePerfdTower}

Keep the notation as above.
We additionally assume the following conditions:
\begin{itemize}
\item
$(R,\fm,k)$ is a Noetherian local ring such that the residue field $k$ is perfect in characteristic $p>0$.

\item
$R$ is a $V$-algebra, where $(V,(p),k)$ is a complete unramified discrete valuation ring in mixed characteristic and its residue field $k$ coincides with the residue field of $R$.
Here, $V$ is isomorphic to the Witt ring $W(k)$ of $k$.

\item
$p^nR \neq 0$ for any $n>0$ (i.e. the characteristic of $R$ is zero).
\end{itemize}

Then it is known that one can construct a perfectoid tower arising from $(V,(p))$.
Indeed, by adjoining the $p$-th power roots of $p$, we obtain a perfectoid tower
\begin{equation}\label{DVRtower}
\xymatrix{
V_0 := V \ar[r]^{t'_0} & V_1 := V[p^{1/p}] \ar[r]^{t'_1} & V_2 := V[p^{1/p^2}] \ar[r]^{t'_2} &\cdots .
}
\end{equation}

\begin{theorem}\label{pTorFreePerfdTowers}
Keep the notation given above.
Suppose that the following conditions hold:
\begin{enumerate}
\item[(i)]
$R$ is $p$-torsion-free (i.e. $p$ is a non-zero-divisor).


\item[(ii)] $\overline{R} := R/pR$ is reduced.
\end{enumerate}

Then the following assertions hold.

\begin{enumerate} 
\item
The tower $(\{S_i\}_{i \geq 0}, \{ s_i \}_{i \geq 0}) := (\{ R_i \otimes_V V_i\}_{i \geq 0}, \{(\tau_R)_i \otimes_V t_i'\}_{i \geq 0})$ is a perfectoid tower arising from $(R,(p))$.

\item
The tilt of $(\{S_i\}_{i \geq 0}, \{ s_i \}_{i \geq 0})$ is isomorphic to the tower
\begin{equation}\label{PerdfTowerTensor}
\xymatrix{
\overline{R_0}\otimes_{\overline{V}} V^{s.\flat} \ar[r] & \overline{R_1} \otimes_{\overline{V}} V_1^{s.\flat} \ar[r] & \overline{R_2} \otimes_{\overline{V}} V_2^{s.\flat} \ar[r] & \cdots
}
\end{equation}
where $\overline{V}:= V/pV$ is the residue field $k$.
\end{enumerate}
\end{theorem}

\begin{proof}
    (1):
    In the following, we denote $R_i/pR_i$ (resp. $V_i /pV_i$) by $\overline{R_i}$ (resp. $\overline{V_i}$). 
    The axiom $(a)$ is trivial. 
    Note that $\overline{(\tau_R)_i}$ is injective by the injectivity of $\overline{\phi_R}$. 
    Since $\overline{R_i}$ and $\overline{V_i}$ are flat $\overline{V}$-algebras,
    \[
    \overline{R_i} \otimes_{\overline{V}} \overline{V_i} \xrightarrow{\overline{(\tau_R)_i} \otimes \textnormal{id}_{\overline{V_i}}} \overline{R_{i+1}} \otimes_{\overline{V}} \overline{V_i} \xrightarrow{\textnormal{id}_{\overline{R_{i+1}}}\otimes_{\overline{V}} \overline{t'_i}  } \overline{R_{i+1}} \otimes_{\overline{V}} \overline{V_{i+1}}
    \]
    is injective and it coincides with $\overline{s_i}$.
    Hence, the axiom (b) holds.
    Since $\overline{(\tau_R)_i}$ and $\overline{t'_i}$ obviously contain the image of each Frobenius endomorphism, it implies that axiom (c) holds.
    Let $F_i : \overline{R_{i+1}} \to \overline{R_i}$ (resp. $F'_i : \overline{V_{i+1}} \to \overline{V_i}$) be the $i$-th Frobenius projection.
    Since both Frobenius projections $F_i$ and $F_i'$ are surjective, the axiom (d) holds.
    Also, the axiom (e) and the axiom (g) are obvious by assumption.
    Hence, it remains to show that the axiom $(f)$ holds.
    We want to prove that the kernel of $\id \otimes_{\overline{V}} F'_i : \overline{R_{i-1}} \otimes_{\overline{V}} \overline{V_i} \to \overline{R_{i-1}} \otimes_{\overline{V}} \overline{V_{i-1}} $ is generated by $1 \otimes_{\overline{V}} \overline{p^{1/p}}$ because $F_i \otimes_{\overline{V}} \id_{\overline{V_i}} : \overline{R_{i}} \otimes_{\overline{V}} \overline{V_i} \to \overline{R_{i-1}} \otimes_{\overline{V}} \overline{V_{i}}$ is injective.
    Since we have the short exact sequence
    \begin{equation*}
    0 \to \overline{R_{i-1}} \otimes_{\overline{V}} (\overline{p^{1/p}}) \to \overline{R_{i-1}}\otimes_{\overline{V}} \overline{V_i} \xrightarrow{ \id_{\overline{R_{i-1}}}\otimes F'_i}  \overline{R_{i-1}}\otimes_{\overline{V}}\overline{V_{i-1}} \to  0,
    \end{equation*}
    the kernel of $\id_{\overline{R_{i-1}}}\otimes_{\overline{V}} F'_i$ is $\overline{R_{i-1}} \otimes_{\overline{V}} (\overline{p^{1/p}}) = (1 \otimes_{\overline{V}} \overline{p^{1/p}})$.

    (2):
    Let $\Phi_0^{(i)}$ be the $0$-th projection map $V_i^{s.\flat} \to \overline{V_i}$.
    Then for any $i\geq 0$, we have the following commutative ladder
    \[
    \xymatrix{
    0 \ar[r] & \Ker(\id_{\overline{R_{i+1}}} \otimes_{\overline{V}} \Phi^{(i+1)}_0 ) \ar[r] \ar[d]^{\textnormal{Res}} & \overline{R_{i+1}} \otimes_{\overline{V}} V^{s.\flat}_{i+1} \ar[rr]^{\id_{\overline{R_{i+1}}} \otimes_{\overline{V}} \Phi^{(i+1)}_0} \ar[d]^{\overline{\phi_i} \otimes_{\overline{V}} (F'_i)^{s.\flat}}  & & \overline{R_{i+1}} \otimes_{\overline{V}} \overline{V_{i+1}} \ar[r] \ar@{->>}[d]^{F_i \otimes_{\overline{V}} F'_i} & 0 \\ 
    0 \ar[r] & \Ker(\id_{\overline{R_{i}}} \otimes_{\overline{V}} \Phi^{(i)}_0 ) \ar[r] & \overline{R_{i}} \otimes_{\overline{V}} V^{s.\flat}_i \ar[rr]^{\id_{\overline{R_{i}}} \otimes_{\overline{V}} \Phi^{(i)}_0 } & & \overline{R_{i}}\otimes_{\overline{V}} \overline{V_i} \ar[r] & 0.
    }
    \]
    where $\textnormal{Res}$ is the restriction of $\overline{\phi_i} \otimes_{\overline{V}} (F'_i)^{s.\flat}$.
    Since $\overline{R_{i+1}}$ is a flat $\overline{V}$-algebra, $\Ker(\id_{\overline{R_{i}}} \otimes_{\overline{V}} \Phi^{(i)}_0)$ is isomorphic to $\overline{R_i} \otimes_{\overline{V}}\Ker(\Phi^{(i)}_0) = \overline{R_i}\otimes_{\overline{V}} ((p^{s.\flat})^{1/p})$.
    Here we have 
    \[
    \vpl \big\{\cdots \xrightarrow{\textnormal{Res}}\Ker(\id_{\overline{R_{i+1}}} \otimes_{\overline{V}} \Phi^{(i+1)}_0 ) \xrightarrow{\textnormal{Res}}\Ker(\id_{\overline{R_{i}}} \otimes_{\overline{V}} \Phi^{(i)}_0 )\big\} \cong \bigcap_{n \geq 0} (\overline{R_i} \otimes_{\overline{V}} (p^{s.\flat})^{p^{n-1}}) = 0,
    \]
    where the last equality follows from Krull's intersection theorem.
    Therefore we obtain the isomorphism $S_i^{s.\flat} \cong \overline{R_i} \otimes_{\overline{V}} V_i^{s.\flat}$, as desired.
\end{proof}

As a corollary, we obtain an explicit description of a perfectoid tower arising from a homomorphic image of the formal power series ring over a complete discrete valuation ring.
Before the proof, we describe the setting of the corollary.

\begin{discussion}\label{SettingCor}
Let $(W(k), (p), k)$ be a complete discrete valuation whose residue field is perfect.
We set $A := W(k)\llbracket x_1,\ldots,x_d \rrbracket$.
Then $A$ is a $\delta$-ring equipped with the $p$-derivation $\delta : A \to A$ such that $\delta(x_i) =0$ for any $i = 1,\ldots, d$.
Let $\phi$ be the Frobenius lift on $A$ defined by $\delta$, let $J=(f_1,\ldots,f_r)$ be a $\phi$-stable ideal of $A$, and we set $R := A/J$.
Then it is easy to prove that $A_i \cong W(k)\llbracket x_1^{1/p^i}, \ldots, x_d^{1/p^i}\rrbracket$ where $A_i$ is the finite colimit defined in \Cref{pdelta0}.
Also, for an element $f= \sum a_{(k_1,\ldots,k_n)}x_1^{k_1}\cdots x_d^{k_d} \in A$, we set $f^{1/p^i}:= \sum a_{(k_1,\ldots,k_n)} x_1^{k_1/p^i}\cdots x_d^{k_d/p^i} \in A_i$.
Then $J_i = \phi_{i,0}^{-1}(J)$ is equal to $(f_1^{1/p^i}, \ldots, f_r^{1/p^i})$ because $\phi_{i,0}$ is injective.
Therefore $R_i$ is isomorphic to $A_i/(f_1^{1/p^i},\ldots,f_r^{1/p^i})$ by \Cref{IsomRi}.
\end{discussion}

\begin{corollary}
\label{p-torFreePerfdTower}
Keep the notation as in \Cref{SettingCor}.
Moreover, set $V_i:=W(k)[p^{1/p^i}]$.
Suppose that $R$ satisfies the following conditions:
\begin{enumerate}
\item[(a)] $R$ is $p$-torsion-free.
\item[(b)] $R/pR$ is reduced.
\end{enumerate}
Then the following assertions hold:
\begin{enumerate}
    \item $R$ admits a perfectoid tower $(\{ S_i \}_{i \geq 0}, \{ s_i \}_{i \geq 0})$ given in \Cref{pTorFreePerfdTowers} (1).
    Moreover, it is isomorphic to
    \begin{equation}\label{CompletePerfdTower}
    R \to (A_1/J_1)[p^{1/p}] \to (A_2/J_2)[p^{1/p^2}] \to \cdots.
    \end{equation} 
    \item The tilt of the perfectoid tower (\ref{CompletePerfdTower}) is isomorphic to
    \[
    \overline{A}\llbracket t \rrbracket/ \overline{J}(\overline{A}\llbracket t \rrbracket) \to \big(\overline{A_1}\llbracket t \rrbracket/ \overline{J_1}(\overline{A_1}\llbracket t \rrbracket)\big)[t^{1/p}] \to \big(\overline{A_2}\llbracket t \rrbracket/ \overline{J_2}(\overline{A_2}\llbracket t \rrbracket)\big)[t^{1/p^2}] \to \cdots.
    \]
    where $\overline{A_i} = A_i/pA_i \cong k\llbracket x^{1/p^i}_1, \ldots, x_d^{1/p^i} \rrbracket$, $\overline{J_i}$ is the image of the ideal $J_i$ of $A_i$ in $\overline{A_i}$, and $\overline{J}_i(\overline{A_i}\llbracket t \rrbracket)$ is the image of the ideal $\overline{J}_i$ of $\overline{A_i}$ in $\overline{A_i}\llbracket t \rrbracket$.
    \item Suppose that $R/pR$ is $F$-pure and Gorenstein.
    Then $R$ is perfectoid pure.
\end{enumerate}
\end{corollary}

\begin{proof}
    $(1)$:
    This is the special case of \Cref{pTorFreePerfdTowers} by the assumptions (a) and (b).
    Now, since $R_i \cong A_i/J_i$ by the above discussion, each layer $R_i\otimes_{W(k)} W(k)[p^{1/p^i}]$ of the tower (\ref{PerdfTowerTensor}) is isomorphic to $(A_i/J_i)[p^{1/p^i}]$, as desired.

    $(2)$:
    Note that $\overline{V} = k$ and $V_i^{s.\flat}$ is isomorphic to $k\llbracket t \rrbracket[t^{1/p^i}]$.
    Hence we have the isomorphisms
    \[
    S_i^{s.\flat} \cong k\llbracket x^{1/p^i}_1,\ldots, x^{1/p^i}_d \rrbracket \otimes_k k\llbracket t^{1/p^i} \rrbracket/J_i \otimes_k k\llbracket t^{1/p^i} \rrbracket \cong \big(\overline{A_i}\llbracket t \rrbracket /\overline{J_i} (\overline{A_i}\llbracket t \rrbracket)\big)[t^{1/p^i}].
    \]
    This implies that the assertion (2) holds.
    
    $(3)$:
    To prove this, we show that $(\tau_R)_{0,i}\otimes_V (t')_{0,i}$ is pure. 
    Since $(\tau_R)_{0,i}:= (\tau_R)_i \circ \cdots \circ (\tau_R)_0 : R_0 \to R_i$ is pure by applying \Cref{deltasplit}, we have the injection 
    $M \to M\otimes_R R_i$ for any $R$-module $M$.
    Moreover, since $V \to V_i$ is faithfully flat (namely, pure) where $V = W(k)$ and $V_i := W(k)[p^{1/p^i}]$, the homomorphism $M \to M\otimes_R R_i \otimes_V V_i$ is injective as a $V$-module.
    Hence it is also injective as an $R$-module.
    This implies that $(\tau_R)_{0,i}\otimes_V (t')_{0,i}$ is pure.
    Therefore the assertion follows from \Cref{PerfdPureTower}.
\end{proof}

The following examples are not log-regular and not quotients of monomial ideals.
\begin{example}\label{ExamplePerfdTowerp-torfree}
    \begin{enumerate}
    \item
    Set $A := W(k) \llbracket a,b,x,y,z,w \rrbracket$ and $R:=A/(ab, xy-zw)$.
    Then $R$ is $p$-torsion-free and $R/pR$ is reduced.
    Also, fix the Frobenius lift $\phi$ induced by the $p$-derivation $\delta$ defined as $\delta(a)=\delta(b)=\delta(x)= \delta(y)=\delta(z)=\delta(w)=0$.
    Then, by \Cref{ExamplePhiStables} (2), we already know that $(ab, xy-zw)$ is $\phi$-stable. 
    Hence $R$ has a perfectoid tower.
    Let us construct it by applying \Cref{p-torFreePerfdTower}.
    Since $\phi$ maps each determinants to its $p$-th power, each $A_i$ is isomorphic to $W(k)\llbracket a^{1/p^i}, b^{1/p^i}, x^{1/p^i}, y^{1/p^i}, z^{1/p^i}, w^{1/p^i} \rrbracket$.
    Moreover, one can directly prove that $I_i = (a^{1/p^i}b^{1/p^i}, x^{1/p^i}y^{1/p^i}-z^{1/p^i}w^{1/p^i})$ because $\phi$ is injective.
    Hence $R_i$ is isomorphic to $W(k)\llbracket a^{1/p^i}, b^{1/p^i}, x^{1/p^i}, y^{1/p^i}, z^{1/p^i}, w^{1/p^i} \rrbracket/(a^{1/p^i}b^{1/p^i}, x^{1/p^i}y^{1/p^i}-z^{1/p^i}w^{1/p^i})$.
    Finally, by tensoring $R_i$ with $V_i := W(k)[1/p^i]$, we obtain the perfectoid tower 
    \[
    R_0 = R \to R_1\otimes_V V_1 \to R_2 \otimes_V V_2 \to \cdots
    \]
    whose layer $R_i\otimes_V V_i$ is isomorphic to 
    \[
    W(k)[1/p^i]\llbracket a^{1/p^i}, b^{1/p^i}, x^{1/p^i}, y^{1/p^i}, z^{1/p^i}, w^{1/p^i} \rrbracket/(a^{1/p^i}b^{1/p^i}, x^{1/p^i}y^{1/p^i}-z^{1/p^i}w^{1/p^i}).
    \]
    Moreover, applying \Cref{pTorFreePerfdTowers} (2), the $i$-th layer of their tilt is 
    \[
    k\llbracket t^{1/p^i}, a^{1/p^i}, b^{1/p^i}, x^{1/p^i}, y^{1/p^i}, z^{1/p^i}, w^{1/p^i} \rrbracket/(a^{1/p^i}b^{1/p^i}, x^{1/p^i}y^{1/p^i}-z^{1/p^i}w^{1/p^i}).
    \]
    
    \item
    Set $A := W(k) \llbracket x,y,z \rrbracket$ where the characteristic of $k$ is not equal to 2 and set $R:=A/(xy-z^2,x^2-y^2)$.
    Since $(xy-z^2,x^2-y^2)$ is a $\phi$-stable ideal, the Frobenius lift on $A$ induces the Frobenius lift on $R$.
    Moreover, since we have $R/pR \cong k \llbracket x,y,z \rrbracket/(xy-z^2, x^2-y^2)$, $R/pR$ is reduced.
    It is obvious that $R$ is $p$-torsion-free.
    Hence $R$ admits a perfectoid tower.
    Moreover, since we can compute $I_i = (x^{1/p^i}y^{1/p^i}-z^{2/p^i}, x^{2/p^i} - y^{2/p^i})$, each $S_i$ is isomorphic to $W(k)[p^{1/p^i}]\llbracket x^{1/p^i},y^{1/p^i},z^{1/p^i} \rrbracket/(x^{1/{p^i}}y^{1/p^i}-z^{2/p^i} ,x^{2/p^i}-y^{2/p^i} )$.
    Also, the tilt $R_i^{s.\flat}$ is isomorphic to $k\llbracket x^{1/p^i},y^{1/p^i},z^{1/p^i} \rrbracket/(x^{1/p^i}y^{1/p^i}-z^{2/p^i}, x^{2/p^i} - y^{2/p^i})\otimes_{\overline{V}} V_i^{s.\flat}$, which is isomorphic to $k\llbracket t^{1/p^i},x^{1/p^i},y^{1/p^i},z^{1/p^i} \rrbracket/(x^{1/p^i}y^{1/p^i}-z^{2/p^i}, x^{2/p^i} - y^{2/p^i})$.
    \end{enumerate}
\end{example}

%

\subsection{Perfectoid towers of  determinantal rings of $2\times 2$-minors}\label{SubSectionDet}

By applying \Cref{deltasplit} and \Cref{pTorFreePerfdTowers}, we show that certain ladder determinantal rings are perfectoid pure.

Let $k$ be a perfect field, let $W(k)$ be its Cohen ring, let $A = W(k)\llbracket X_1,\ldots, X_d \rrbracket$ be a formal power series ring, and let $M = (M_{ij})$ be an $m\times n$ matrix of squarefree monomials with respect to $X_1,\ldots,X_d$ or $0$.
Then the determinantal ideal of $2\times 2$-minors $I_2(M)$ is generated by monomials or binomials.
We set $A_i = A[X_1^{1/p^i},\ldots,X_d^{1/p^i}]$ and
\[
M^{1/p^i} = 
\left(
\begin{matrix}
M_{11}^{1/p^i} & M_{12}^{1/p^i} & \cdots & M_{1n}^{1/p^i} \\ M_{21}^{1/p^i} & M_{22}^{1/p^i} & \cdots & M_{2n}^{1/p^i} \\
\vdots & \vdots & \ddots & \vdots \\
M_{m1}^{1/p^i} & M_{m2}^{1/p^i} & \cdots & M_{mn}^{1/p^i} 
\end{matrix} 
\right).
\]
We also set $I := I_2(M)$ and $R = A/I$.
Then the ring $R_i$ defined in \Cref{Rpdelta0} is isomorphic to $A_i/I_2(M^{1/p^i})$ by \Cref{IsomRi}.
Now, let $S_i$ and $s_i : S_i \to S_{i+1}$ be as defined in \Cref{pTorFreePerfdTowers} (1). Then one can construct a perfectoid tower as follows.

\begin{corollary}\label{DetPerfd}
    Keep the notation as above.
    Suppose that $R/pR$ is reduced.
    Then the following assertions hold.
    \begin{enumerate}
    \item $(\{ S_i \}_{i \geq 0}, \{ s_i \}_{i \geq 0})$ is a perfectoid tower, which is isomorphic to 
    \[
    A/I_2(M) \to \big(A_1/I_2(M^{1/p})\big)[p^{1/p}] \to \big(A_2/I_2(M^{1/p^2})\big)[p^{1/p^2}] \to \cdots.
    \]
    \item The tilt of the tower is isomorphic to
    \[
    \overline{A}\llbracket t \rrbracket/\overline{I_2(M)}(\overline{A}\llbracket t \rrbracket) \to \big(\overline{A_1}\llbracket t \rrbracket/ \overline{I_2(M^{1/p})}(\overline{A_1}\llbracket t \rrbracket)\big) [t^{1/p}] \to \big(\overline{A_2}\llbracket t \rrbracket/ \overline{I_2(M^{1/p^2})}(\overline{A_2}\llbracket t \rrbracket)\big)[t^{1/p^2}] \to \cdots.
    \]
    \end{enumerate}
\end{corollary}

Let $k$ be a field and let $X=(X_{ij})$ be the $m \times n$ matrix of indeterminants over $k$.
Then a subset $Y$ of $X$ is called a \textit{ladder} if whenever $X_{ij}, X_{hk} \in Y$ and $i \leq h$, $j\leq k$, then $X_{ih}, X_{hj} \in Y$.
For a ladder $Y$, let $k[Y]$ be the polynomial ring $k[X_{ij}~|~X_{ij} \in Y ]$.
Fix a positive integer $t$ greater than $1$.
Then we denote $k[Y]/I_t(Y)$ by $R_t(Y)$ where $I_t(Y)$ is the ideal generated by the $t\times t$-minors of $Y$.
Here, the ideal $I_t(Y)$ is called a \textit{ladder determinantal ideal} and the ring $R_t(Y)$ is called a \textit{ladder determinantal ring}.
It is known that ladder determinantal rings are Cohen--Macaulay normal domains \cite{BC03, Nar86}.

In characteristic $p>0$, Conca and Herzog proved that for a wide Gorenstein ladder $Y$, the ladder determinantal ring $R_t(Y) = k[Y]/I_t(Y)$ over a perfect field $k$ is $F$-pure and they conjectured that any ladder determinantal ring is $F$-pure in their paper \cite{CH97}.
Recently, De Stefani--Monta{\~n}o--Betancourt--Seccia--Varbaro proved the conjecture.
See \cite[Corollary 4.3]{DSMNBSV26}.

Also, the ladder determinantal ideal of $2\times 2$ minors $I_2(Y)$ is generated by binomials.
Hence $I_2(Y)$ is $\phi$-stable.
By combining these with our results, one can construct an example of perfectoid pure singularities.

\begin{corollary}
\label{ExamplePerfdPure}
    Let $k$ be a perfect field in characteristic $p>0$ and let $Y$ be a ladder of size $m\times n$.
    Let $\widehat{R_2(Y)} = k\llbracket Y \rrbracket/I_2(Y)$ be the completion of the ladder determinantal ring of $2$-minors over $k$.
    Then the following assertions hold.
    \begin{enumerate}
        \item The lift $\mathcal{R}_2(Y) := W(k)\llbracket Y \rrbracket/I_2(Y)$ admits a perfectoid tower arising from $(\mathcal{R}_2(Y), (p))$.
        \item If $\mathcal{R}_2(Y)$ is Gorenstein, then it is perfectoid pure.
    \end{enumerate}
\end{corollary}

\begin{proof}
    The assertion (1) follows from \Cref{DetPerfd}.
    Note that $\mathcal{R}_2(Y)$ is a Gorenstein local domain and $\mathcal{R}_2(Y)/p\mathcal{R}_2(Y) \cong R_2(Y)$ is $F$-pure by \cite[Corollary 4.3]{DSMNBSV26}.
    Hence we obtain the assertion (2) by \Cref{DeformPerfdPurity}, as desired.
\end{proof}

\begin{example}
    Let $Y$ be the following ladder:
    \[
    \begin{matrix}
        X_{11} & X_{12} & X_{13}  \\
        X_{21} & X_{22} & X_{23} \\
        X_{31} &  X_{32} & 0 
    \end{matrix}
    \]
    This ladder is wide by definition.
    The criterion of Conca \cite[Proposition 2.5]{Con95} implies the associated ladder determinantal ring $R_2(Y)$ is Gorenstein.
    By the criterion of Glassbrenner and Smith \cite[Theorem 7]{GS95}, $R_2(Y)$ is not a complete intersection.
    Hence the lift $\mathcal{R}_2(Y) = W(k)\llbracket Y \rrbracket/I_2(Y)$ is a non-complete intersection Gorenstein perfectoid pure domain by \Cref{ExamplePerfdPure}.
\end{example}

Inspired by the proof of \cite[Theorem 4.24]{BMPSTWW24}, we prove the following result.

\begin{proposition}
Let $(R,\fm)$ be a complete Cohen--Macaulay local domain with perfect residue field $k$ of mixed characteristic with a regular sequence $f_1,\ldots,f_n \in \fm$.
Set $S:=R/(f_1,\ldots,f_n)$. 
Assume that
$$
R_0:=R \to R_1 \to R_2 \to \cdots
$$
is a perfectoid tower arising from $(R,(p))$ such that $R_i$ is a domain for all $i \ge 0$ with $R_\infty^\wedge$ being the $p$-adic completion of $R_\infty:=\varinjlim_{i \ge 0}R_i$. Set $S^{R_\infty^\wedge}_{\perfd}$ to be the perfectoidization of $(R_\infty^\wedge \otimes_R S)^\wedge$ (see \cite[Corollary 7.3 and Theorem 7.4]{BS22} for the structure on the perfectoidization). Then $S^{R_\infty^\wedge}_{\perfd}$ is a big Cohen-Macaulay $R$-algebra in the sense that $H_\fm^i(S^{R_\infty^\wedge}_{\perfd})=0$ for all $i< \dim S$. 
Finally, the conclusion holds when $R$ is a local log-regular ring or a ladder determinantal ring isomorhic to $\mathcal{R}_2(Y)$ as in Corollary \ref{ExamplePerfdPure}.
\end{proposition}


\begin{proof}
Although the same proof given in \cite[Theorem 4.24]{BMPSTWW24} works, we provide the proof for the sake of readers. 
By \cite[Corollary 3.52]{INS23}, $R_\infty^\wedge$ is a perfectoid ring.
We first show that $R_i$ is Cohen--Macaulay.
Fix a generator $p^{s.\flat}$ of $I_0^{s.\flat}$.
Since $p \in R_0$ and $p^{s.\flat} \in R_0^{s.\flat}$ are non-zero-divisors, $R_0^{s.\flat}$ is Cohen--Macaulay.
We note that $R_i^{s.\flat}$ is isomorphic to $(R_0^{s.\flat})^{1/p^i}$, which is finitely generated $R_0^{s.\flat}$-module.
Hence $R_i^{s.\flat}$ is Cohen--Macaulay.
Finally, since $p$ (resp. $p^{s.\flat}$) is also a non-zero-divisor on $R_i$ (resp. $R_i^{s.\flat}$), $R_i$ is Cohen--Macaulay.

Since $R_i \to R_{i+1}$ is a module-finite extension by \Cref{FinitePerfdTower}, we see that $R_\infty$ is a balanced big Cohen-Macaulay $R$-algebra. 
Then $R_\infty^\wedge$ is a balanced big Cohen-Macaulay $R$-algebra by \cite[Theorem 0.1]{Ye18}. 
By \cite[Theorem 2.3.4]{KS20}, we can find a $p$-completely faithfully flat extension $R_\infty^\wedge \to T$ such that $T$ is a perfectoid ring and all elements of $R_\infty^\wedge$ admit $p$-power roots in $T$. Then note that $f_1,\ldots,f_n$ can be extended to a full system of parameters of $R$ as a regular sequence $f_1,\ldots,f_n,f_{n+1},\ldots,f_d$ and
$$
R_\infty^\wedge/(f_1^{1/p^m},\ldots,f_n^{1/p^m})R_\infty^\wedge \to T/(f_1^{1/p^m},\ldots,f_n^{1/p^m})T
$$
is also $p$-completely faithfully flat for all $m>0$. In particular, after taking colimit, it follows that $f_{n+1},\ldots,f_d$ is a regular sequence on $T/(f_1^{1/p^\infty},\ldots,f_n^{1/p^\infty})=\varinjlim_{m>0}T/(f_1^{1/p^m},\ldots,f_n^{1/p^m})$. On the other hand, we claim that
\begin{equation}
\label{p-completeflat}
S^{R_\infty^\wedge}_{\perfd} \cong R_\infty/\big((f_1,\ldots,f_n)R_\infty\big)_\perfd \to T/\big((f_1^{1/p^\infty},\ldots,f_n^{1/p^\infty})T\big)_\perfd
\end{equation}
is $p$-completely faithfully flat. Indeed, let $\big((f_1^{1/p^\infty},\ldots,f_n^{1/p^\infty})T\big)^-$ be the $p$-adic closure of the ideal $(f_1^{1/p^\infty},\ldots,f_n^{1/p^\infty})T$. Then by \cite[Lemma 2.19]{CLMST25}, we have
$$
\big((f_1^{1/p^\infty},\ldots,f_n^{1/p^\infty})T\big)_\perfd \cong \big(\big((f_1,\ldots,f_n)R_\infty\big)_\perfd T\big)_\perfd \cong \big((f_1^{1/p^\infty},\ldots,f_n^{1/p^\infty})T\big)^-,
$$
which shows that $(\ref{p-completeflat})$ is faithfully flat modulo $p$, thus $p$-completely faithfully flat. Since $H_\fm^i(T/(f_1^{1/p^\infty},\ldots,f_n^{1/p^\infty})^-) \cong H_\fm^i(T/(f_1^{1/p^\infty},\ldots,f_n^{1/p^\infty}))=0$ for $i< \dim R$, the $p$-complete faithful flatness of $(\ref{p-completeflat})$ gives the claim. 

Finally, in the case that $R$ is log-regular, the claimed perfectoid tower $R_0:=R \to R_1 \to R_2 \to \cdots$ is constructed in \cite[Proposition 3.58]{INS23} and each $R_i$ is log-regular, thus a Cohen-Macaulay normal domain.
\end{proof}

\subsection{Perfectoid towers arising from a section ring of a smooth projective variety}\label{SubSecSecRing}

Our aim is to construct perfectoid towers by applying the geometric results from \cite{IS24}. First we prove some results on singularities of certain rings of sections in mixed characteristic. Similar treatments also appear in \cite{BMPSTWW24}. Let $k$ be an algebraically closed field of characteristic $p>0$ and let $W(k)$ be the ring of Witt vectors. We start with the following construction.

Let \(X\) be an integral regular scheme such that $X$ is a projective $A$-scheme for a Noetherian ring \(A\) and let us fix an ample \(\mathbb{Q}\)-divisor \(D\) on \(X\). Define the graded ring
\begin{equation*}
R(X, D) \coloneqq \bigoplus_{k=0}^\infty H^0(X,\mathcal{O}_X(kD))
\end{equation*}
which is called the \emph{(generalized) section ring} of \(X\) with respect to \(D\) following \cite[Definition 4.1]{Sh17}. Here, $\mathcal{O}_X(nD)$ is defined as $\mathcal{O}_X(\lfloor nD \rfloor)$, where $\lfloor nD \rfloor$ is an integral divisor which is a round-down of $nD$. Thus, $\mathcal{O}_X(nD)$ is an invertible sheaf because $X$ is a regular scheme. $R(X,D)$ is a Noetherian subring of the function field $K(X)$. In particular, this is a Noetherian integral domain. If \(L = \mathcal{O}_X(D)\), then we often denote \(R(X, D)\) by \(R(X, L)\), namely,
\begin{equation*}
R(X, L) \coloneqq \bigoplus_{k=0}^\infty H^0(X,L^k).
\end{equation*}
This is a finitely generated over \(A\) by \citeSta{0B5T} and there is an isomorphism \(X \cong \Proj(R(X, L))\).

\begin{lemma}
\label{ModpReductionSectionRings}
Let \(X\) be a smooth projective variety over \(k\). Assume that \(X\) admits a projective flat lifting \(\mathcal{X}\) over \(W(k)\).
  Let \(\mathcal{L}\) be an ample line bundle on \(\mathcal{X}\) such that \(L \coloneqq \mathcal{L}|_X\) satisfies \(H^1(X, L^k) = 0\) for all \(k \geq 1\).\footnote{By \citeSta{0892}, \(L\) is an ample line bundle on \(X\). Therefore, after replacing it with sufficiently higher power, we may assume that $H^1(X,L^k)=0$ for all $k \ge 1$ by Serre's ampleness criterion (see, for example, \citeSta{0B5T}).} Then there is an isomorphism
\begin{equation}
\label{ample2}
R(\mathcal{X}, \mathcal{L}) \otimes_{W(k)} k \cong R(X, L)
\end{equation}
of graded rings. Moreover, if \(\mathcal{X}\) has a Frobenius lift \(\widetilde{F}_X\) and \(\mathcal{L}\) satisfies \((\widetilde{F}_X)^*\mathcal{L} \cong \mathcal{L}^p\), then there exists a Frobenius lift \(\phi_{\mathcal{X}, \mathcal{L}}\) on \(R(\mathcal{X}, \mathcal{L})\):
\begin{equation}
\label{amplebundle2}
\phi_{\mathcal{X},\mathcal{L}} \colon R(\mathcal{X},\mathcal{L}) \to R(\mathcal{X},\mathcal{L})
\end{equation}
which sends the $k$-th graded component in the domain to the $pk$-th graded component in the target. Note that $\phi_{\mathcal{X},\mathcal{L}}$ is compatible with the isomorphism (\ref{ample2}) and the Witt-vector Frobenius map on $W(k)$.
\end{lemma}

\begin{proof}
The vanishing of \(H^1(X, L^k)\) and \cite[Remark 8.3.11.2 and Corollary 8.3.11]{I05} shows that \(\mathcal{L}^k\) is cohomologically flat in degree \(0\) for all \(k \geq 1\). In particular, we have $H^0(\mathcal{X},\mathcal{L}^k) \otimes_{W(k)} k \cong H^0(X,L^k)$ for all $k \geq 1$. By taking the global sections associated to the short exact sequence $0 \to \mathcal{O}_{\mathcal{X}} \xrightarrow{\times p} \mathcal{O}_{\mathcal{X}} \to \mathcal{O}_X \to 0$, the induced $k$-linear map $\alpha:\Gamma(\mathcal{X},\mathcal{O}_{\mathcal{X}}) \otimes_{W(k)} k \to \Gamma(X,\mathcal{O}_X)$ is injective. Since $X$ is an integral variety over $k=\overline{k}$, we have $\Gamma(X,\mathcal{O}_X)=k$, which gives bijectivity of $\alpha$. The above discussions combine together to conclude the desired isomorphism \eqref{ample2}.

Assume that \(\mathcal{X}\) has a Frobenius lift \(\widetilde{F}_X\) and \(\mathcal{L}\) satisfies \((\widetilde{F}_X)^*\mathcal{L} \cong \mathcal{L}^p\). We see that the structure map $\widetilde{F}_X^\sharp \colon \mathcal{O}_{\mathcal{X}} \to (\widetilde{F}_X)_*\mathcal{O}_{\mathcal{X}}$ induces
\begin{equation}
\label{amplebundle}
\mathcal{L} \to ((\widetilde{F}_X)_*\mathcal{O}_{\mathcal{X}}) \otimes_{\mathcal{O}_{\mathcal{X}}} \mathcal{L} \cong (\widetilde{F}_X)_*(\widetilde{F}_X)^*\mathcal{L} \cong (\widetilde{F}_X)_*\mathcal{L}^p,
  \end{equation}
  where the first isomorphism follows from the projection formula. Since $\widetilde{F}_X$ is an affine morphism by Lemma \cite[Lemma 2.10]{IS24}, we see that $(\ref{amplebundle})$ induces mappings $H^0(\mathcal{X},\mathcal{L}^k) \to H^0(\mathcal{X},\mathcal{L}^{pk})$ for $k \ge 0$, which then yields a graded map of graded rings
  \begin{equation*}
  \phi_{\mathcal{X},\mathcal{L}}:R(\mathcal{X},\mathcal{L}) \to R(\mathcal{X},\mathcal{L}^p) \hookrightarrow R(\mathcal{X},\mathcal{L})
  \end{equation*}
  which sends the $k$-th graded component in the domain to the $pk$-th graded component in the target. This completes the proof.
\end{proof}

First, we show that the ring of sections $R(\mathcal{X},\mathcal{L})$ possesses nice properties. See \cite{ShTa24} and \cite{I23} for relevant discussions.

\begin{proposition}
\label{deltaring}
Let \(X\) be a smooth projective variety \(X\) over an algebraically closed field \(k\) of characteristic \(p > 0\). Assume that \(X\) admits a quasi-canonical lifting \((\mathcal{X}, \widetilde{F}_X)\). Then there is a choice of an ample line bundle $\mathcal{L}$ relative to $\mathcal{X}/\Spec(W(k))$ such that the Noetherian $\mathbb{N}$-graded ring $R(\mathcal{X},\mathcal{L}):=\bigoplus_{k=0}^\infty H^0(\mathcal{X},\mathcal{L}^k)$ satisfies the following properties.
\begin{enumerate}
\item
There is an isomorphism of $\mathbb{N}$-graded rings $R(\mathcal{X},\mathcal{L})/pR(\mathcal{X},\mathcal{L}) \cong R(X,\mathcal{L}|_X)$.
    
\item
$R(\mathcal{X},\mathcal{L})$ is a normal domain and $R(X,\mathcal{L}|_X)$ is an (\(F\)-finite) $F$-split normal domain. 
    
\item
$R(\mathcal{X},\mathcal{L})$ carries a graded ring map $\phi_{\mathcal{X},\mathcal{L}}$ which lifts the Frobenius map on $R(\mathcal{X},\mathcal{L})/pR(\mathcal{X},\mathcal{L})$.
    
\item
Suppose moreover that $\Omega^1_X$ is trivial. Then $R(\mathcal{X},\mathcal{L})$ is perfectoid injective (see \Cref{def:perfectoidpureinj} (2) for the definition of perfectoid injectivity).
  
\item
If $X$ is an ordinary Abelian variety,\footnote{If \(X\) is an ordinary Abelian variety, this construction is already known. For example, \cite[Lemma 4.11]{KT24}.} then $R(X,\mathcal{L}|_X)$ is not Cohen--Macaulay (in particular, it is not a splinter).
\end{enumerate} 
\end{proposition}


\begin{proof}
Fix an ample line bundle \(L\) on \(X\). By \cite[Lemma 2.13]{IS24}, there exists a line bundle \(\mathcal{L}\) on \(\mathcal{X}\) with \(L = \mathcal{L}|_X\) such that \((\widetilde{F}_X)^*\mathcal{L} \cong \mathcal{L}^p\). By \cite[Proposition 1.41]{KM98} and \cite[Lemma 2.11]{IS24}, $\mathcal{L}$ is relatively ample on $\mathcal{X}/W(k)$. As stated in \Cref{ModpReductionSectionRings}, we can assume that \(H^1(X, L^k) = 0\) for all \(k \geq 1\) by replacing \(\mathcal{L}\) with sufficiently higher power. We show that this \(\mathcal{L}\) is a required one.

The claimed isomorphism in the assertion $(1)$ was already proved in \Cref{ModpReductionSectionRings}. Moreover, since $X$ has a quasi-canonical lifting over $W(k)$ by \cite[Corollary 4.4]{IS24}, it is globally Frobenius-split in view of \cite[Theorem 5.5]{Z17}, the ring of sections $R(X,L)$ is an $F$-split normal domain which is finitely generated over \(k\). It is a general fact that the ring of sections of a normal integral projective scheme is a normal domain, which gives the assertion $(2)$. Let us check the assertion $(3)$. Namely, it suffices to show that $\phi_{\mathcal{X}, \mathcal{L}}$ is a lift of the Frobenius morphism. The commutative diagram
\[
\xymatrix@M=10pt{ 
\mathcal{X}  &\mathcal{X} \ar[l]^{\widetilde{F}_X} \\
X \ar[u]&\ar[l]^{F_X} X \ar[u]\\
}
\]
induces the commutative diagram of graded rings
\[
\xymatrix@M=10pt{ 
R(\mathcal{X},\mathcal{L}) \ar[r]^{\phi_{\mathcal{X},\mathcal{L}}} \ar[d] & R(\mathcal{X},\mathcal{L}) \ar[d] \\
R(X,L) \ar[r]^{\overline{\phi}_{\mathcal{X},\mathcal{L}}} & R(X,L) \\
}
\]
Here, $\overline{\phi}_{\mathcal{X},\mathcal{L}}$ is the Frobenius endomorphism on $R(X,L)$. Then by $(\ref{ample2})$, we deduce that $\phi_{\mathcal{X},\mathcal{L}} \pmod{p}=\overline{\phi}_{\mathcal{X},\mathcal{L}}$ as claimed. One can check that $R(\mathcal{X},\mathcal{L})$ is perfectoid injective by invoking \cite[Proposition 7.4 and Example 7.7]{BMPSTWW24} since \(\omega_{\mathcal{X}}\) is trivial. This shows (4).
Finally, as for the assertion $(5)$, the non-Cohen--Macaulayness of $R(X,\mathcal{L}|_X)$ is proved in the following manner. By \cite[\S 13 Corollary 2, p. 121]{Mumford}, we get $H^k(X,\mathcal{O}_X) \ne 0$ for all $ 0 \le k \le \dim X$. On the other hand, it follows from \cite[Chapitre III Corollaire 2.1.4]{Gro1} that
$$
H^{i+1}_{\fm}(R(X,L)) \cong \bigoplus_{n=0}^\infty H^i(X,L^{n})
$$
for all $i \ge 1$, where $\fm$ is the graded maximal ideal $R(X,L)_{>0}$. In particular, $H^2_\fm(R(X,L)) \ne 0$. Since we are assuming $\dim X \ge 2$, we get $\dim R(X,L) \ge 3$. This implies that the non-vanishing of $H^2_\fm(R(X,L))$ forces $R(X,L)$ to be non-Cohen--Macaulay.  
\end{proof}




In view of \Cref{deltaring}, we can apply the construction of \Cref{pdelta0}, where
$$
\delta_{\mathcal{X},\mathcal{L}}(x)=\frac{\phi_{\mathcal{X},\mathcal{L}}(x)-x^p}{p}
$$
to the $\delta$-ring $(R(\mathcal{X},\mathcal{L}),\delta_{\mathcal{X},\mathcal{L}})$. So we get the towers of graded rings $(\{R(\mathcal{X},\mathcal{L})_i\}_{i \geq 0},\{\tau_i\}_{i \ge 0})$ and $(\{R(X,L)_i\}_{i \ge 0},\{\overline{\tau}_i\}_{i \ge 0})$. Since $X$ is an integral smooth variety, it follows that $R(X,L)$ is a normal domain. We remark that $\overline{\tau}_i$ is purely inseparable and the Frobenius map on $R(X,L)_{i}$ factors as $R(X,L)_{i} \twoheadrightarrow R(X,L)_{i-1} \xrightarrow{\overline{\tau}_i} R(X,L)_{i}$, where the first surjective map is the $p$-th power map. There is the commutative diagram of towers:
\begin{equation}
\label{perfectcomm3}
\vcenter{\xymatrix{
R(\mathcal{X},\mathcal{L}) \ar[d]^{\pi_0} \ar[r] &R(\mathcal{X},\mathcal{L})_1\ar[d]^{\pi_1} \ar[r] & \cdots \ar[r] & R(\mathcal{X},\mathcal{L})_n \ar[d]^{\pi_n} \ar[r] &\cdots \\
R(X,L) \ar[r] & R(X,L)_1 \ar[r] & \cdots \ar[r] & R(X,L)_n \ar[r] & \cdots \\
}}
\end{equation}

Fix a compatible sequence $\{p^{1/p^i}\}_{i \geq 0}$ of $p$-power roots of $p$ in the algebraic closure of $W(k)[1/p]$.
Let $V_i\coloneqq W(k)[p^{1/p^i}]$ which is a complete ramified discrete valuation ring. We define
\begin{equation}
\label{perfectoidtower}
R_i\coloneqq R(\mathcal{X},\mathcal{L})_i^{\wedge} \otimes_{W(k)}V_i \cong R(\mathcal{X},\mathcal{L})_i^{\wedge}[p^{1/p^i}],
\end{equation}
where the completion is taken along the graded maximal ideal $(p)+R^+(\mathcal{X},\mathcal{L})_{i}$, where $R^+(\mathcal{X},\mathcal{L})_{i}$ is the strictly positive part of $R(\mathcal{X},\mathcal{L})_{i}$. This completed ring $R(\mathcal{X}, \mathcal{L})$ is also the same as the completion of the $R(\mathcal{X},\mathcal{L})$-module $R(\mathcal{X},\mathcal{L})_i$ in the $(p)+R^+(\mathcal{X},\mathcal{L})$-adic topology, as can be checked after the reduction modulo $p$. 
Let us prove some ring-theoretic properties of $R(\mathcal{X}, \mathcal{L})$.

\begin{lemma}
\label{deriveddelta2}
Let $A$ be a Noetherian ring with an ideal $I \subseteq A$ and let $\widehat{A}$ be the $I$-adic completion of $A$. Let $\{M_0 \to M_1 \to \cdots \to M_n\}$ be a finite directed system of $A$-linear maps of finitely generated $A$-modules. Let $\{\widehat{M_0} \to \widehat{M_1} \to \cdots \to \widehat{M_n}\}$ be the system induced by the $I$-adic completion. Then there is an isomorphism of $\widehat{A}$-modules:
$$
\widehat{\varinjlim_{i}M_i} \cong \varinjlim_{i}\widehat{M_i}.
$$
In particular, for each $i \ge 0$, there is a canonical ring isomorphism $R(\mathcal{X}, \mathcal{L})_i^{\wedge} \cong \widehat{R(\mathcal{X}, \mathcal{L})}_i$.
\end{lemma}

\begin{proof}
The finite colimit $\varinjlim_{i}M_i$ is calculated as the quotient $(\bigoplus_{i=0}^nM_i)/L$, where $L$ is an $A$-submodule of $\bigoplus_{i=0}^nM_i$. In particular, $\varinjlim_{i}M_i$ is a finitely generated $A$-module. Then
$$
\widehat{\varinjlim_{i}M_i} \cong \widehat{(\bigoplus_{i=0}^nM_i)/L} \cong (\widehat{\bigoplus_{i=0}^nM_i})/\widehat{L} \cong 
(\bigoplus_{i=0}^n\widehat{M_i})/\widehat{A}L \cong \varinjlim_{i}\widehat{M_i},
$$
where one notices that $\widehat{A}L$ is the $I$-adic completed $\widehat{A}$-module $\widehat{L}$. To deduce that $R(\mathcal{X}, \mathcal{L})_i^{\wedge}$ is isomorphic to $ \widehat{R(\mathcal{X}, \mathcal{L})}_i$, it suffices to note that $\delta_{\mathcal{X},\mathcal{L}}:R(\mathcal{X},\mathcal{L}) \to R(\mathcal{X},\mathcal{L})$ is module-finite and we are done.
\end{proof}

\begin{lemma}
\label{RingProperties}
$R(\mathcal{X},\mathcal{L})_i^{\wedge}$ is a $(d+2)$-dimensional complete Noetherian local normal domain for each $i \geq 0$, in particular, flat over $W(k)$. Furthermore, each $R_i = R(\mathcal{X},\mathcal{L})_i^{\wedge} \otimes_{W(k)} V_i$ is also a complete Noetherian normal local domain with residue field $k$.
\end{lemma}

\begin{proof}
Set the graded maximal ideal $I \coloneqq (p) + R^+(\mathcal{X}, \mathcal{L})$ of $R(\mathcal{X}, \mathcal{L})$.
Since all rings and modules in question are Noetherian, the derived completion coincides with the classical completion, so there is no distinction between them. By \Cref{deriveddelta2}, we have a canonical isomorphism of rings
$$
R(\mathcal{X}, \mathcal{L})_i^{\wedge} \cong \widehat{R(\mathcal{X}, \mathcal{L})}_i.
$$
Since $X$ is a normal integral projective variety, the ring of sections $R(X,L)$ is a normal graded domain and its completion $\widehat{R(X,L)}$ is a complete Noetherian normal local domain of characteristic $p>0$, where the completion is the $R^+(X, L)$-adic completion. Also, $\widehat{R(X, L)}_i$ is a complete Noetherian normal local domain since it is isomorphic to $\widehat{R(X, L)}$ as a ring.
After taking the $I$-adic completion, we have a chain of isomorphisms:
$$
\widehat{R(\mathcal{X},\mathcal{L})}/p\widehat{R(\mathcal{X},\mathcal{L})}\cong \widehat{R(\mathcal{X},\mathcal{L})/pR(\mathcal{X},\mathcal{L})} \cong \widehat{R(X,L)}.
$$
We can apply \cite[Theorem 8.4]{Mat86} to conclude that $\widehat{R(\mathcal{X},\mathcal{L})}_i$ is a complete Noetherian local ring flat over $W(k)$. Moreover, $\widehat{R(\mathcal{X},\mathcal{L})}_i$ is a normal domain by the deformation invariance of normality (see \cite{Mu22} and \cite{Ch82}). The flat base change $R_i$ is a finite extension of $\widehat{R(\mathcal{X}, \mathcal{L})}_i$ and thus $R_i$ is a complete Noetherian local domain. By setting $x=p^{1/p^i}$ in \Cref{normalcriterion} below, we know that $R_i$ is a normal domain.
\end{proof}

We used the following lemma above.

\begin{lemma}
\label{normalcriterion}
Let $A$ be an integral domain with a regular element $x \in A$. If $A/xA$ is reduced and $A[1/x]$ is an integrally closed domain, then $A$ is an integrally closed domain.
\end{lemma}

\begin{proof}
Assume that $a \in \Frac(A)$ is integral over $A$. Then since $A[1/x]$ is an integrally closed domain, it follows that $a \in A[1/x]$ and we can write $a=b/x^n$ for some $b \in A$ and $n \ge 0$. Assume that $n$ is the smallest possible. If $n=0$, then there is nothing to prove. So assume that $n>0$. Let
$
f(t)\coloneqq t^m+a_1t^{m-1}+\cdots+a_{m-1}t+a_m=0
$
be a monic equation such that $f(a)=0$. Then we get
$
b^m+a_1b^{m-1}x^n+\cdots+a_{m-1}b(x^n)^{m-1}+a_m(x^n)^m=0,
$
which shows that $\overline{b}^m \in A/xA$. Since $A/xA$ is reduced, we must get $b \in xA$. So $b=xz$ for $z \in A$. However, we get
$
a=b/x^n=(yz)/x^n=z/x^{n-1},
$
contradicting the minimality of $n$. Hence $n=0$.
\end{proof}

The next result provides us a non-trivial example of a perfectoid tower $\{R_i\}_{i \ge 0}$ in which each $R_i$ is a non-Cohen--Macaulay normal domain.

\begin{theorem}
\label{delta-structure}
Assume that $X$ is a $d$-dimensional smooth projective variety defined over an algebraically closed field $k$ of characteristic $p>0$. Then the tower $R_0 \to R_1 \to R_2 \to \cdots$ constructed as in $(\ref{perfectoidtower})$ enjoys the following properties.
\begin{enumerate}
\item
The tower $ R_0 \to R_1 \to R_2 \to \cdots $ is a perfectoid tower.

\item
For each $i \ge 0$, $R_i$ is a $d+2$-dimensional Noetherian complete normal local domain with residue field $k$. $R_i \to R_{i+1}$ is a module-finite extension and $p^{1/p^i} \in R_i$.

\item
The initial ring $R\coloneqq R_0$ arises as a completion of a certain $W(k)$-flat graded normal domain which is obtained by taking the $p$-adic deformation of the section ring of a smooth projective $k$-variety $X$ such that there is an ordinary Abelian variety $A$ over $k$ and a surjective finite \'etale morphism $A \to X$ of degree prime to $p$. Moreover, if one takes $X$ to be an ordinary Abelian variety of dimension at least 2, then $R$ is not Cohen--Macaulay. In this situation, all $R_i$'s are not Cohen--Macaulay.
\end{enumerate}
\end{theorem}

\begin{proof}
$(1)$: 
By \cite[Lemma 2.17]{BS22}, $R(\mathcal{X},\mathcal{L})_i^{\wedge}$ admits a Frobenius lift $\phi$ induced by that of $R(\mathcal{X},\mathcal{L})_i$.
By \Cref{FrobLiftInjpZariskian}, \Cref{deltaring}, and \Cref{RingProperties}, one can check that the tower $R_0 \to R_1 \to R_2 \to \cdots$ satisfies the conditions (i), (ii), (iii) in \Cref{pTorFreePerfdTowers}.
Hence it is a perfectoid tower.
 More explicitly, the tower is isomorphic to
 \[
 R(\mathcal{X},\mathcal{L})^{\wedge} \to R(\mathcal{X},\mathcal{L})_1^{\wedge}[p^{1/p^1}] \to R(\mathcal{X},\mathcal{L})_2^{\wedge}[p^{1/p^2}] \to \cdots \to R(\mathcal{X},\mathcal{L})_i^{\wedge}[p^{1/p^i}] \to \cdots.
 \]

$(2)$: By \Cref{RingProperties}, it suffices to check that $R_i \to R_{i+1}$ is a module-finite extension but this is a direct consequence of the $F$-finite property of $R(X, L)$ and Nakayama's lemma.

$(3)$: Let us construct a perfectoid tower with desired properties by means of the cone of a smooth projective variety (see \cite[Example 3.6, Example 3.7, and Corollary 3.11]{K13} for some methods). Now let $X$ be an ordinary Abelian variety over an algebraically closed field $k$ of dimension $\ge 2$. By \Cref{deltaring}, the pair $(X,F_X)$ admits a canonical lifting $(\mathcal{X},\widetilde{F}_X)$. For an appropriate choice of an ample line bundle $L$ on $X$, we obtain a pair $(R(\mathcal{X},\mathcal{L}),\phi_{\mathcal{X},\mathcal{L}})$. Let us check that $R(\mathcal{X},\mathcal{L})$ is a graded normal domain that is not Cohen--Macaulay. This fact can be checked as follows. By \cite[Corollary 2, p. 121]{Mumford}, we have the non-vanishing $H^k(X,\mathcal{O}_X) \ne 0$ for any $0 \le k \le \dim X$. Then by sheaf/local cohomology correspondence \cite[Corollaire 2.1.4]{Gro1}, we see that $R(X,L)$ is not Cohen--Macaulay. Since these properties are preserved under completion and scalar extension in $(\ref{perfectoidtower})$, we can construct the perfectoid tower $\{R_i\}_{i \ge 0}$. There are several ways to check that each $R_i$ is not Cohen--Macaulay. Here is one proof. Suppose to the contrary that $R_i$ is Cohen--Macaulay for some $i>0$. Since $R=R_0 \to R_i$ is module-finite, we find that $R_i$ is a maximal Cohen--Macaulay $R$-algebra. But this is a contradiction to \cite[Theorem 3.18]{ShTa24}.
\end{proof}

\begin{remark}
For an ordinary Abelian variety $(X,F_X)$, we have a canonical lifting $(\mathcal{X},\widetilde{F}_X)$ and the section ring $R(\mathcal{X},\mathcal{L})$ has a $\delta$-ring structure. This extends to the $\delta$-ring structure on the complete local domain $\widehat{R(\mathcal{X},\mathcal{L})}$. By \Cref{LimCMSeq}, we have a lim Cohen--Macaulay sequence $\{\widehat{R(\mathcal{X},\mathcal{L})_n}\}_{n \geq 0}$. On the other hand, it follows from \cite[Theorem 3.18]{ShTa24} that $\widehat{R(\mathcal{X},\mathcal{L})}$ admits a maximal Cohen--Macaulay module $M$ (recall that $\widehat{R(\mathcal{X},\mathcal{L})}$ is not Cohen--Macaulay if $\dim X \ge 2$). Then the constant system $M=M_0=M_1=M_2=\cdots$ provides another lim Cohen--Macaulay sequence, which can be easily checked from \Cref{LimCMDef}.
\end{remark}

\begin{corollary}\label{SectionTilts}
The tilt of the tower $R=R_0 \to R_1 \to R_2 \to \cdots$ associated to $(p)$ is isomorphic to the perfect tower: 
$$
R(X, L)\llbracket T \rrbracket \to R(X, L)_{1}\llbracket T^{1/p}\rrbracket \to R(X, L)_{2}\llbracket T^{1/p^2}\rrbracket\to \cdots. 
$$
Moreover, each $R_i$ is an $F$-pure normal domain.
\end{corollary}

\begin{proof}
The ring $V_i^{s.\flat}$ is isomorphic to $k\llbracket T^{1/p^i} \rrbracket := (k\llbracket T \rrbracket)[T^{1/p^i}]$
Hence the tilt of the tower given in \Cref{pTorFreePerfdTowers} (2) is isomorphic to
\[
R(X, L)\llbracket T \rrbracket \to R(X, L)_{1}\llbracket T^{1/p}\rrbracket \to R(X, L)_{2}\llbracket T^{1/p^2}\rrbracket\to \cdots. 
\]
Also, each $R(X,L)_i\llbracket T^{1/p^i} \rrbracket$ is isomorphic to $(R(X,L)\llbracket T \rrbracket)^{1/p^i}$ because the tilt is a perfect tower.
Since the $F$-purity carries over from the origin to its finite perfections, one readily sees that $R(X,L)_i\llbracket T^{1/p^i} \rrbracket$ is $F$-pure, as desired.
\end{proof}

\begin{remark}
\Cref{delta-structure} (1) and \Cref{SectionTilts} are also shown in Ishizuka's paper \cite{I23}.
\end{remark}

\subsection{Perfectoid towers arising from quotients by squarefree monomial ideals with $p$-torsion elements}\label{SubSubSectMonomialPerfd}

We construct an example of perfectoid towers that have $p$-torsion elements. This example is similar to the notion of Stanley--Reisner rings in mixed characteristic introduced by Olivia Strahan. See \cite{Str25}.
Before constructing such perfectoid towers, let us organize certain properties of monomial ideals with respect to regular system of parameters in regular local rings.

\begin{definition}
    Let $(R,\fm,k)$ be a regular local ring and let $\underline{x} := x_1, \ldots, x_d$ be a regular system of parameters of $R$.
    Then any product $\mathbf{x} := x_1^{a_1}\cdots x_d^{a_d} \in R$ is called a \textit{monomial} of $R$.
    Moreover, a monomial $\mathbf{x} := x_1^{a_1}\cdots x_d^{a_d}$ is called \textit{squarefree} if $a_i \in \{ 0,1\}$ for any $1 \leq i \leq d$.
\end{definition}

Fix a regular local ring and its system of parameters.
Then an ideal generated by monomials (resp. squarefree monomials) is called a \textit{monomial ideal} (resp. a \textit{squarefree monomial ideal}).

\begin{lemma}\label{SquareFreeMonRed}
    Let $(R,\fm,k)$ be a regular local ring.
    Then a squarefree monomial ideal is an intersection of prime monomial ideals.
    In particular, a monomial ideal is squarefree if and only if it is reduced. 
\end{lemma}

\begin{proof}
    The second assetion directly follows from the first assetion.
    Hence, we prove the first assertion.
    Let $I$ be a squarefree monomial ideal with respect to a system of parameters $x_1,\ldots,x_d$.
    We will use the induction on the total number of $x_1,\ldots,x_d$ appearing as monomials in the generators of $I$.
    We denote this number by $G(I)$.
    The case $G(I)=1$ is obvious.
    Assume that the case $G(I)>1$ holds.
    We may assume that $x_1$ appears as an element in generators of $I$.
    Set $I=(\mathbf{x}_1,\ldots,\mathbf{x}_n,\mathbf{y}_{1},\ldots,\mathbf{y}_{m})$, where $\mathbf{x}_i = x_1x_2^{\varepsilon_{i2}}\cdots x_d^{\varepsilon_{id}}$ and $\mathbf{y}_j = x_2^{\varepsilon_{j2}}\cdots x_d^{\varepsilon_{jd}}$ for some $\varepsilon_{ij} \in \{0, 1\}$.
    Also, set $\mathbf{x}_i' := \mathbf{x}_i/x_1$.
    Then we claim that there is the decomposition $I = (I,x_1) \cap(I, \mathbf{x}_1')$.
    Indeed, the inclusion $I \subseteq (I,x_1)\cap(I,\mathbf{x}_1')$ holds obviously.
    Conversely, pick $y \in (I,x_1)\cap(I, \mathbf{x}_1')$.
    Then there exist elements $c \in I$ and $d \in R$ such that $y=c+dx_1  \in (I,\mathbf{x}_1')$.
    Since $dx_1 =y-c$ lies in $(I,\mathbf{x}_1')$, we obtain $dx_1=\sum_{i=1}^nc_i\mathbf{x}_i + \sum_{j=1}^m d_j \mathbf{y}_j +e\mathbf{x}_1'$ for some $c_i, d_j, e \in R$.
    Hence we have $x_1(d-\sum_{i=1}^n c_i\mathbf{x}'_i) =  \sum_{j=1}^m d_j \mathbf{y}_{j} +e\mathbf{x}_1' \in (\mathbf{y}_1,\ldots,\mathbf{y}_m,\mathbf{x}_1') =: J$.
    By the induction assumption, $J$ is an intersection of prime monomial ideals  generated by monomials not involving $x_1$.
    Hence we obtain $d-\sum_{i=1}^n c_i\mathbf{x}'_i \in J$, which implies $dx_1 =\sum_{i=1}^n c_i\mathbf{x}_i + \sum_{j=1}^m d'_j \mathbf{y}_{j} +e'\mathbf{x}_1 \in I$ for $d'_j, e' \in R$.
    Therefore we obtain $y \in I$.
    
    Finally, since the inequalities $G((I,x_1)), G((I,\mathbf{x}_1')) < G(I)$ hold, we can apply the induction assumption to $(I,x_1)$ and $(I, \mathbf{x}_1')$.
    Hence, $I$ is an intersection of prime monomial ideals, as desired.
\end{proof}

In this subsection, we assume the following conditions.
\begin{itemize}
    \item $A := W(k)\llbracket x_2, \ldots, x_d \rrbracket $ where $k$ is a perfect field and we denote $p$ by $x_1$.
    \item $\phi : A \to A$ is the Frobenius lift induced by the $p$-derivation $\delta : A \to A$ such that $\delta (x_i) =0$ for every $2 \leq i\leq d$.
    \item $I = (\mathbf{x}_1,\ldots,\mathbf{x}_n )$ where $\mathbf{x}_i = x_1^{\epsilon_{i1}}x_2^{\epsilon_{i2}} \cdots x_d^{\epsilon_{id}}$ and $\epsilon_{ij} \in \{0,1\}$ for any $1 \leq i \leq n$ and $1 \leq j \leq d$.
    \item $A_i$ is the direct system defined in \Cref{pdelta0}. 
    \item $I_i$ is the ideal of $A_i$ defined in \Cref{Ipdelta0}.
    \item $R_i$ is the direct system defined in \Cref{Rpdelta0}.
\end{itemize}

Since each $V_i$ appearing in (\ref{DVRtower}) is a flat $V$-algebra, we obtain the short exact sequence
\[
0 \to I_i \otimes _V V_i \to A_i \otimes_V V_i \to R_i \otimes_V \V_i \to 0.
\]
Hence $(R_i \otimes _V V_i)_{\textnormal{red}}$ is isomorphic to $A_i\otimes_V V_i/\sqrt{I_i \otimes_V V_i}$.

Moreover, the tower
\begin{equation}\label{pFreeTowers}
0 \to R = R_0 \to R_1\otimes_V V_1 \to R_2 \otimes_V V_2 \to \cdots 
\end{equation}
induces the tower
\begin{equation}\label{pFreeTowersRed}
0 \to R_{\textnormal{red}} = (R_0)_{\textnormal{red}} \to (R_1\otimes_VV _1)_{\textnormal{red}} \to (R_2\otimes_V V_2)_{\textnormal{red}} \to \cdots.
\end{equation}

Here, we set $S_i := (R_i\otimes_V V_i)_{\textnormal{red}}$.

\begin{lemma}\label{MonomialIsom}
    $S_i$ is isomorphic to $A[x_1^{1/p^i}, x_2^{1/p^i}, \ldots x_d^{1/p^i}]/(\mathbf{x}_1^{1/p^i},\ldots, \mathbf{x}_n^{1/p^i})$ where $\mathbf{x}_j$ is the monomial $x_1^{\epsilon_{j1}/p^i}x_2^{\epsilon_{j2}/p^i}\cdots x_d^{\epsilon_{jd}/p^i}$.
\end{lemma}

\begin{proof}
We know that $A_i\otimes_V V_i$ is isomorphic to $V_i\llbracket x^{1/p^i}_2\ldots, x^{1/p^i}_d \rrbracket$.
Hence it suffices to show that $\sqrt{I_i \otimes_V V_i}$ is isomorphic to $(\mathbf{x}_1^{1/p^i},\ldots, \mathbf{x}_n^{1/p^i})$.
We note that $I_i \otimes_V V_i$ is isomorphic to $\left((\mathbf{x}^*_1)^{1/p^i}, \ldots, (\mathbf{x}^*_n)^{1/p^i}\right)$ where $(\mathbf{x}^*_j)^{1/p^i}$ is the monomial $x_1^{\epsilon_{j1}}x_2^{\epsilon_{j2}/p^i}\cdots x_d^{\epsilon_{jd}/p^i}$.
For any $1 \leq j \leq n$, we have $(\mathbf{x}^*_j)^{1/p^i} = x_j^{\epsilon_{j1}(p^i-1)/p^i}\mathbf{x}^{1/p^i}_j \in (\mathbf{x}_1^{1/p^i},\ldots, \mathbf{x}_n^{1/p^i})$.
Moreover, the ideal $(\mathbf{x}_1^{1/p^i},\ldots, \mathbf{x}_n^{1/p^i})$ is generated by squarefree monomials with respect to the regular system of parameters $x_1^{1/p^i},\ldots, x_d^{1/p^i}$ in $A_i\otimes_V V_i$.
This implies that $\sqrt{I_i\otimes_V V_i}$ is isomorphic to $(\mathbf{x}_1^{1/p^i},\ldots, \mathbf{x}_n^{1/p^i})$ by \Cref{SquareFreeMonRed}, as desired.
\end{proof}

\begin{lemma}
\label{Monomialptor}
    By substituting elements, we may assume that $\epsilon_{j1}=1$ for $1\leq j \leq l$ and $\epsilon_{j1}=0$ for $l+1 \leq j \leq n$.
    Then $(S_i)_{p\textnormal{-tor}} = (\mathbf{y}_1^{1/p^i},\ldots,\mathbf{y}_l^{1/p^i})$, where $\mathbf{y}_j^{1/p^i} := \mathbf{x}_j^{1/p^i}/x_1^{1/p^i}$. 
\end{lemma}

\begin{proof}
The inclusion $(S_i)_{p\textnormal
-tor} \supseteq (\mathbf{y}_1^{1/p^i},\ldots,\mathbf{y}_l^{1/p^i})$ is obvious.
Pick an element $f \in (S_i)_{p\textnormal{-tor}}$.
By the isomorphism given in \Cref{MonomialIsom}, there exist elements $g_1, \ldots, g_n \in V_i\llbracket x^{1/p^i}_2\ldots, x^{1/p^i}_d \rrbracket$ such that $(pf)^n = g_1\mathbf{x}_1^{1/p^i}+\cdots +g_n\mathbf{x}_n^{1/p^i}$ in $V_i\llbracket x^{1/p^i}_2\ldots, x^{1/p^i}_d \rrbracket$ for some $n \geq 0$.
Hence we obtain the equation
$$
p^{1/p^i}(p^{{np^i-1}/p^i}f^n - g_1\mathbf{y}^{1/p^i}_1-g_2\mathbf{y}^{1/p^i}_2-\cdots-g_l\mathbf{y}^{1/p^i}_l) = g_{l+1}\mathbf{x}^{1/p^i}_{l+1} + \cdots + g_n\mathbf{x}_{n}^{1/p^i}.
$$
The left-hand side is divided by $p^{1/p^i}$. Hence $g_{l+1}, \ldots, g_n$ are divided by $p^{1/p^i}$ and we obtain the equation
$$
p^{{np^i-1}/p}f^n = g_1\mathbf{y}^{1/p^i}_1 + g_2\mathbf{y}^{1/p^i}_2 + \cdots + g_l\mathbf{y}^{1/p^i}_l + g'_{l+1}\mathbf{x}^{1/p^i}_{l+1} + \cdots + g'_n\mathbf{x}_{n}^{1/p^i}
$$
for some $g'_{l+1}, \ldots, g'_{n} \in V_i\llbracket x^{1/p^i}_2\ldots, x^{1/p^i}_d \rrbracket$.
Moreover by applying the same argument, there exist elements $g_1'', \ldots, g_n'' \in V_i\llbracket x^{1/p^i}_2\ldots, x^{1/p^i}_d \rrbracket$ such that 
$$
f^n = g''_1\mathbf{y}^{1/p^i}_1 + g''_2\mathbf{y}^{1/p^i}_2 + \cdots + g''_l\mathbf{y}^{1/p^i}_l + g''_{l+1}\mathbf{x}^{1/p^i}_{l+1} + \cdots + g''_n\mathbf{x}_{n}^{1/p^i}.
$$
and the image of $f^n = g''_1\mathbf{y}^{1/p^i}_1 + g''_2\mathbf{y}^{1/p^i}_2 + \cdots + g''_l\mathbf{y}^{1/p^i}_l$ in $S_i$ is contained in $(\mathbf{y}_1^{1/p^i},\ldots,\mathbf{y}_l^{1/p^i})$.
Hence $f$ is contained in $\sqrt{(\mathbf{y}_1^{1/p^i},\ldots,\mathbf{y}_l^{1/p^i})} = (\mathbf{y}_1^{1/p^i},\ldots,\mathbf{y}_l^{1/p^i})$, as desired.
\end{proof}

We finally provide an example of a perfectoid tower that has $p$-torsion elements.


\begin{proposition}\label{pTorsionMonomialTower}
The tower (\ref{pFreeTowersRed}) is a perfectoid tower that arises from $(R,(p))$.
\end{proposition}

\begin{proof}
        The axioms (a) and (e) are trivial.
        Notice that we have the isomorphisms 
        \[
        S_i/pS_i \cong 
        k \llbracket \overline{p^{1/p^i}}, \overline{x_2^{1/p^i}},\ldots,\overline{x_d^{1/p^i}} \rrbracket/(\overline{\mathbf{x}_1^{1/p^i}}, \ldots, \overline{\mathbf{x}_n^{1/p^i}}) \cong 
        k\llbracket t, x_2, \ldots, x_d \rrbracket/(t^{p^i}, \tilde{\mathbf{x}}_1,\ldots, \tilde{\mathbf{x}}_n),
        \]
        where $\tilde{\mathbf{x}}_j = t^{\epsilon_{j1}}x_2^{\epsilon_{j2}}\cdots x_d^{\epsilon_{jd}}$.
        Also, we have the following ring homomorphisms
        \[
        k\llbracket t, x_2, \ldots, x_d \rrbracket/(t^{p^i}, \tilde{\mathbf{x}}_1,\ldots, \tilde{\mathbf{x}}_n)
        \cong k\llbracket t^p, x_2^p, \ldots, x_d^p \rrbracket/(t^{p^{i+1}}, \tilde{\mathbf{x}}_1^p, \ldots, \tilde{\mathbf{x}}_n^p)
        \hookrightarrow k\llbracket t, x_2,\ldots, x_d  \rrbracket/(t^{p^{i+1}}, \tilde{\mathbf{x}}_1, \ldots, \tilde{\mathbf{x}}_n).
        \]
        Since the last of the compositions is isomorphic to $S_{i+1}/pS_{i+1}$, we obtain the injectivity of $S_i/pS_i \to S_{i+1}/pS_{i+1}$. 
        In summary, we obtain the following commutative diagram:
        \[
        \xymatrix{
        S_i/pS_i \ar[r] \ar[d]^{\cong} & S_{i+1}/pS_{i+1} \ar[d]^{\cong} \\
        k\llbracket t, x_2, \ldots, x_d \rrbracket/(t^{p^i}, \tilde{\mathbf{x}}_1, \ldots, \tilde{\mathbf{x}}_n) \ar@{^(->}[r] & k\llbracket t, x_2,\ldots, x_d  \rrbracket/(t^{p^{i+1}}, \tilde{\mathbf{x}}_1, \ldots, \tilde{\mathbf{x}}_n).
        }
        \]
        Hence the axiom (b) holds.
        The axioms (c) and (d) follow from the fact that the composite map
        \[
        S_{i+1}/pS_{i+1} \xrightarrow{\cong} k\llbracket t, x_2,\ldots, x_d  \rrbracket/(t^{p^{i+1}}, \tilde{\mathbf{x}}_1, \ldots, \tilde{\mathbf{x}}_n) \twoheadrightarrow k\llbracket t, x_2, \ldots, x_d \rrbracket/(t^{p^i}, \tilde{\mathbf{x}}_1, \ldots, \tilde{\mathbf{x}}_n) \xleftarrow{\cong} S_i/pS_i,
        \]
        is the Frobenius projection $F_i : S_{i+1}/pS_{i+1} \to S_i/pS_i ; a \mapsto a^p$ where the second map is the canonical surjection.
        To show that the axiom (f) holds, it suffices to prove $\Ker (F_i) = (\overline{p^{1/p}})$.
        Pick an element $\overline{a} \in \Ker(F_i)$ such that $a^p = c_1\mathbf{x}_1^{1/p^i} + \cdots + c_d\mathbf{x}_{n}^{1/p^i}$ in $V_i \llbracket x_2^{1/p^i}, \ldots, x_d^{1/p^i} \rrbracket/pV_i \llbracket x_2^{1/p^i}, \ldots, x_d^{1/p^i} \rrbracket$.
        Here, since $(a-c_1^{1/p}\mathbf{x}_1^{1/p^{i+1}}-\cdots -c_d^{1/p}\mathbf{x}_d^{1/p^{i+1}})^p=0$ in $V_{i+1} \llbracket x_2^{1/p^{i+1}}, \ldots, x_d^{1/p^{i+1}} \rrbracket/pV_{i+1} \llbracket x_2^{1/p^{i+1}}, \ldots, x_d^{1/p^{i+1}} \rrbracket$ and the kernel of the Frobenius homomorphism on
        $
        V_{i+1} \llbracket x_2^{1/p^{i+1}}, \ldots, x_d^{1/p^{i+1}}\rrbracket /pV_{i+1} \llbracket x_2^{1/p^{i+1}}, \ldots, x_d^{1/p^{i+1}}\rrbracket
        $
        is generated by $p^{1/p}$, we can express 
        \[
        a = c_1^{1/p}\mathbf{x}_1^{1/p^{i+1}} + \cdots + c_d^{1/p}\mathbf{x}_{n}^{1/p^{i+1}}+dp^{1/p}
        \]
        for some $d \in V_{i+1} \llbracket x_2^{1/p^{i+1}}, \ldots, x_d^{1/p^{i+1}}\rrbracket/pV_{i+1} \llbracket x_2^{1/p^{i+1}}, \ldots, x_d^{1/p^{i+1}}\rrbracket$.
        Hence we obtain $\overline{a} \in (p^{1/p})$ in $S_{i+1}/pS_{i+1}$, as desired.
        Finally, we show that the axiom (g) holds.
        By \Cref{Monomialptor}, $(S_i)_{p\textnormal{-tor}}$ vanishes by multiplying $p$.
        Moreover, let $(F_i)_{\textnormal{tor}} : (S_{i+1})_{p\textnormal{-tor}} \to (S_i)_{p\textnormal{-tor}}$ be the $p$-power map between $p$-torsions. 
        Then this is well-defined and surjective because it is a non-unital ring map by the form of $p$-torsions $(S_i)_{p\textnormal{-tor}}$.
        Also, since $S_i/pS_i$ is reduced and $(S_i)_{p\textnormal{-tor}} \to S_i/pS_i$ is injective, $(F_i)_{\tor}$ is injective.
        Finally, each $(F_i)_{\tor}$ commutes with the $i$-th Frobenius projection $F_i$.
        Hence the axiom (g) holds. 
\end{proof}

\begin{example}
Set $A:= \mathbb{Z}_p\llbracket x,y \rrbracket$ and $I:= (p) \cap(x,y) = (px,py)$. Since $(p)$ and $(x,y)$ are $\phi_A$-stable, so is $I$ by \Cref{IdealOperationsStable}.
This implies that the Frobenius lift $\phi_A$ on A induces the ring map $\phi_R$ on $R:=A/I$ (it is a Frobenius lift on $R$).
Let $R_0 := R \to R_1 \to R_2 \to \cdots $ be a tower constructed as in $(\ref{pFreeTowersRed})$.
Then, by \Cref{pTorsionMonomialTower}, it is perfectoid tower and each $R_i$ is isomorphic to $A[p^{1/p^i}, x^{1/p^i}, y^{1/p^i}]/(p^{1/p^i}x^{1/p^i}, p^{1/p^i}y^{1/p^i})$ by \Cref{MonomialIsom}.
\end{example}

Since each $S_i$ is not $p$-torsion-free, we can not apply \Cref{weakperfectlimCM} directly to determine whether the sequence $\{S_i \}_{i \geq 0}$ in the tower $(\ref{pFreeTowersRed})$ is a lim Cohen--Macaulay sequence or not.


\end{document}